\documentclass[12pt,reqno]{amsart}

\usepackage{fullpage}
\usepackage{amsmath,amsthm,amssymb,amscd,amsfonts,graphicx}
\usepackage{amsrefs}

\numberwithin{equation}{section}
\usepackage{fullpage}

\renewcommand{\vec}[1]{\ensuremath{\mathbf{#1}}}
\newcommand{\beq}{\begin{equation}}
\newcommand{\eeq}{\end{equation}}

\newcommand{\RR}{\ensuremath{\mathbb{R}}}

\newcommand{\bpm}{\begin{pmatrix}}
\newcommand{\epm}{\end{pmatrix}}

\newcommand{\dif}{\ensuremath{\mathrm{d}}}
\newcommand{\mi}{\ensuremath{\mathrm{i}}}

\newcommand{\pd}{\ensuremath{\partial}}

\newcommand{\PP}{\ensuremath{\mathcal{P}}}
\newcommand{\QP}{\ensuremath{\mathcal{Q}}}

\DeclareMathOperator{\dv}{div}
\DeclareMathOperator{\grad}{grad}
\DeclareMathOperator{\re}{Re}

\renewcommand{\Re}{\re}

\newcommand{\lam}{\ensuremath{\lambda}}
\newcommand{\mat}[1]{\ensuremath{\mathbf{#1}}}
\newcommand*{\defequals}{\mathrel{\vcenter{\baselineskip0.5ex \lineskiplimit0pt
                     \hbox{\scriptsize.}\hbox{\scriptsize.}}}%
                     =}

\newcommand\R{\mathbb R}

\newcommand\br{\begin{remark}}
\newcommand\er{\end{remark}}
\newcommand\bp{\begin{pmatrix}}
\newcommand\ep{\end{pmatrix}}
\newcommand{\be}{\begin{equation}}
\newcommand{\ee}{\end{equation}}
\newcommand\ba{\begin{equation}\begin{aligned}}
\newcommand\ea{\end{aligned}\end{equation}}

\newcommand{\bap}{\begin{app}}
\newcommand{\eap}{\end{app}}
\newcommand{\begs}{\begin{exams}}
\newcommand{\eegs}{\end{exams}}
\newcommand{\beg}{\begin{example}}
\newcommand{\eeg}{\end{exaplem}}
\newcommand{\bpr}{\begin{proposition}}
\newcommand{\epr}{\end{proposition}}
\newcommand{\bt}{\begin{theorem}}
\newcommand{\et}{\end{theorem}}
\newcommand{\bc}{\begin{corollary}}
\newcommand{\ec}{\end{corollary}}
\newcommand{\bl}{\begin{lemma}}
\newcommand{\el}{\end{lemma}}
\newcommand{\bd}{\begin{definition}}
\newcommand{\ed}{\end{definition}}
\newcommand{\brs}{\begin{remarks}}
\newcommand{\ers}{\end{remarks}}

\newcommand{\CC}{{\mathbb C}}

\newtheorem{theorem}{Theorem}[section]
\newtheorem{proposition}[theorem]{Proposition}
\newtheorem{corollary}[theorem]{Corollary}
\newtheorem{lemma}[theorem]{Lemma}

\theoremstyle{remark}
\newtheorem{remark}[theorem]{Remark}
\theoremstyle{definition}
\newtheorem{definition}[theorem]{Definition}

\newtheorem{example}[theorem]{Example}

\newcommand{\me}{\ensuremath{\mathrm{e}}}

\newcommand{\spm}{{\ensuremath{{\scriptscriptstyle\pm}}}}
\renewcommand{\sp}{{\ensuremath{{\scriptscriptstyle +}}}}
\newcommand{\sm}{{\ensuremath{{\scriptscriptstyle -}}}}

\title{Euler vs. Lagrange: The role of coordinates in practical Evans-function computations}
\author{Blake Barker}
\address{Department of Mathematics, Brigham Young University, Provo, UT 84602}
\email{blake@mathematics.byu.edu}
\thanks{B.B. was partially supported by NSF grant DMS-1400872.}
\author{Jeffrey Humpherys}
\address{Department of Mathematics, Brigham Young University, Provo, UT 84602}
\email{jeffh@math.byu.edu}
\thanks{J.H. was partially supported by NSF grant DMS-0847074}
\author{Gregory Lyng}
\address{Department of Mathematics, University of Wyoming, Laramie, WY 82071}
\email{glyng@uwyo.edu}
\thanks{G.L. was partially supported by NSF grants DMS-0845127 and DMS-1413273}
\author{Kevin Zumbrun}
\address{Department of Mathematics, Indiana University, Bloomington, IN 47405}
\email{kzumbrun@indiana.edu}
\thanks{K.Z. was partially supported by NSF grant DMS-0801745}
\date{Last Updated:  \today}
\begin{document}
\begin{abstract}
The Evans function has become a standard tool in the mathematical 
study of nonlinear wave stability.
In particular, computation of its zero set gives a convenient numerical method for determining the point spectrum of the associated linear operator (and thus the spectral stability of the wave in question).  We report on an unexpected complication that frustrates this computation for viscous shock profiles in gas dynamics. Although this phenomenon---related to the choice of Eulerian or Lagrangian coordinate system used to describe the gas---is present already in the one-dimensional setting, its implications are especially important in the multidimensional case where no computationally viable Lagrangian description of the gas is readily available. We introduce new ``pseudo-Lagrangian'' coordinates that allow us to overcome this difficulty, and we illustrate the utility of these coordinates in the setting of isentropic gas dynamics in two space dimensions.
\end{abstract}
\maketitle
\tableofcontents
\raggedbottom

\section{Introduction}\label{sec:intro}
\subsection{Overview}\label{ssec:overview}
The modern theory for the stability of nonlinear waves 
employs a combination of tools from functional analysis and from dynamical systems, and the Evans function is a key link between these two mathematical disciplines; see, e.g., \cites{AGJ,KP,S,Z1}. 
In this paper, we describe an unexpected obstacle to Evans-function computations for viscous profiles in gas dynamics.
This obstacle arises from the Eulerian coordinate system used to describe the motion of the gas. While the phenomenon arises even in a single space dimension, it has so far been missed due to the use by practitioners
of the somewhat simpler Lagrangian equations.
However, in multiple space dimensions, Lagrangian coordinates become impractical due to complexity/introduction of spurious modes
\cite{PYZ}, and the issue becomes central \cite{HLyZ3}. Thus, the resolution we describe here---a  set of ``pseudo-Lagrangian'' coordinates---appears to be a crucial component of any successful multidimensional Evans-function computations for viscous shocks in gas dynamics (and related models).

To begin, we briefly describe the abstract mathematical setting in the one-dimensional case. To that end, 
consider a system of conservation laws with viscosity in a single space dimension. This is a system of partial differential equations of the form
\beq\label{eq:claw}
U_t+F(U)_x = (B(U)U_x)_x\,.
\eeq
In system \eqref{eq:claw}, the unknown $U=U(x,t)$ is in $\RR^n$, the flux $F$ is a function from $\RR^n$ to itself, and the viscosity matrix $B$ is an $\RR^{n\times n}$-valued function on $\RR^n$. Our motivating example of such a system is the Navier--Stokes equations of gas dynamics; observe that both the Eulerian formulation \eqref{eq:ns} and the Lagrangian formulation \eqref{eq:Lag} have the form of equation \eqref{eq:claw}. A viscous shock profile is a traveling-wave solution of equation \eqref{eq:claw} connecting constant states $U_\spm$. That is, it is a solution of the form 
\beq\label{eq:tw}
U(x,t)=\bar U(x-st)\,,\quad \lim_{z\to\pm\infty}\bar U(z)=U_\spm\,.
\eeq
By shifting to a moving coordinate frame, we may assume that the speed $s$ is zero. Thus, the (now) standing-wave solution $\bar U(x)$ is a steady solution of equation \eqref{eq:claw}. To investigate the stability of this wave, we first linearize about it to obtain an equation that approximately describes the evolution of a small perturbation $V$:
\beq\label{eq:linearized}
V_t=LV\defequals(B(x)V_x)_x-(A(x)V)_x\,,
\eeq
where 
\[
B(x) \defequals B(\bar U(x))\,,
\quad\text{and}\quad 
A(x)V\defequals \dif F(\bar U(x))V -\dif B(\bar U(x))\big(V, \bar U'(x)\big)\,.
\]
The goal, then, is to determine the point spectrum of the variable coefficient (but asymptotically constant) operator $L$. To that end, we recast the eigenvalue problem $\lambda W=LW$ as a first-order system
\beq\label{eq:eval}
Z'=\mat{A}(x;\lambda)Z\,,
\eeq
where the prime denotes differentiation with respect to the spatial variable $x$, and $Z\in\CC^N$ (the size of $N$ depends on the structure of the system \eqref{eq:claw}). Since the point spectrum of $L$  in the unstable half plane is made up of those values $\lambda_*$ for which there is a nontrivial solution $Z(x;\lambda_*)$ of equation \eqref{eq:eval} which satisfies
\[
\lim_{x\to\pm\infty} Z(x;\lam_*)=0\,,
\]
these values can be detected by the vanishing of a Wronskian $D(\lam)$, known as the Evans function. More precisely, since $\bar U$ tends to constant states as $x\to\pm\infty$, there are limiting matrices
\[
\mat{A}_\spm(\lambda)\defequals\lim_{x\to\pm\infty}\mat{A}(x;\lambda)\,.
\]
For this introductory discussion, we suppose that for $\lam\in\{z\in\CC\,:\,\Re z>0\}$---the unstable half plane, the dimension of the stable subspace $S_\sp$ of $\mat{A}_\sp$ is $k$ and that the dimension of the unstable subspace $U_\sm$ of $\mat{A}_\sm$ is $N-k$. Then, the Evans function is constructed by building analytic (with respect to $\lam$) bases of solutions 
\[
\{z_1^\sp(x;\lam),\ldots,z_k^\sp(x;\lam)\}
\quad 
\text{and}
\quad
\{z_{k+1}^\sm(x;\lam),\ldots,z_N^\sm(x;\lam)\}
\]
spanning the manifolds of solutions of equation \eqref{eq:eval} that tend to zero at each spatial infinity. These bases are built by initializing at the spatial infinities with data from $S_\sp$ and $U_\sm$ and then integrating equation \eqref{eq:eval} toward $x=0$.
Then, the Evans function is defined to be
\beq
D(\lam)\defequals\det(z_1^\sp,\ldots,z_k^\sp,z_{k+1}^\sm,\ldots,z_N^\sm)|_{x=0}\,.
\eeq
It is evident from this construction that a zero of $D$ corresponds to the existence of a solution of equation \eqref{eq:eval} which decays at both spatial infinities, i.e., an eigenfunction. 

It follows that the computation of $D$ (and, in particular, its zero set) is a central component of the stability analysis. However, for even modestly complicated systems in a single space dimension, this is a task that must be done numerically. Fortunately, this is a computational problem that is by now well understood, and a variety of techniques and algorithms appear in the literature. Starting with a system of form \eqref{eq:eval}, the numerical approximation of $D$ essentially consists of two tasks. First, one must compute analytic bases of $S_\sp$ and $U_\sm$. Second, one must solve the differential equation \eqref{eq:eval} on sufficiently large intervals $[0,M_\sp]$ and $[-M_\sm,0]$. There is a kind of stiffness (when $k\neq 1$ and $N-k\neq1$) associated with this second problem due to the need to resolve modes of differing exponential decay (growth) rates in order to track the entire subspace of decaying (growing) solutions. A now standard solution to this problem is to work in the exterior product space so that the desired subspace appears as the single maximally stable (unstable) mode. An early example of this kind of numerical computation for solitary-wave solutions of a Boussinesq-type equation can be found in the paper of Alexander \& Sachs \cite{AS}. For viscous shock profiles, such as discussed above, the program of numerically approximating $D$ using exterior products was initiated and developed by Brin \cites{B_PHD,Br,BrZ}. Bridges and collaborators \cites{AB,BDG} independently rediscovered this method and clarified its relationship to the earlier compound-matrix method of Ng \& Reid \cites{NR1,NR2,NR3,NR4} for stiff ordinary differential equations. Two key later discoveries by Humpherys \& Zumbrun \cite{HuZ2} and by Humpherys, Sandstede, \& Zumbrun \cite{HSZ} helped open the door to large-scale Evans-function computations such as arise in complicated physical problems. The issue is that the exterior-product method, while elegant, does not scale well as $N$ grows. Humpherys \& Zumbrun \cite{HuZ2} proposed an ``analytic orthogonalization'' technique which allows for a much more efficient representation of the growing/decaying subspaces. In related work dealing with the other computational task, Humpherys, Sandstede, \& Zumbrun \cite{HSZ} proposed an efficient numerical algorithm, based on Kato's projection method \cite{Kato}, that is suitable for computing analytic bases of $S_\sp$ and $U_\sm$ when $k$ and $N-k$ are large. (In practice, it is typical that $k\sim N/2$.) More recent developments include alternative approaches to tackle the problem of large systems \cites{LMNT,LMT} and techniques for root-following as parameters vary \cite{HL}.

As the preceding discussion indicates, there is now a robust collection of numerical methods associated with approximating the Evans function. One culmination of this development is the \textsc{STABLAB} package \cite{STABLAB}, a \textsc{MATLAB}-based suite of routines that implements both the exterior-product method and the analytic-orthogonalization method (among other features). Using \textsc{STABLAB}, computational Evans-function techniques have been applied to gas dynamics in one space dimension \cites{BHRZ,BHLRZ,HLZ,HLyZ}, combustion in one space dimension \cites{HLyZ2,HHLZ,BHLZ1}, and magnetohydrodynamics in one space dimension \cite{BHZ}. A recent development is the use of rigorous numerical calculations to establish numerical proofs of spectral stability \cites{B,BZ}. This latter development is of particular interest since spectral stability---more precisely, a condition stated in terms of an Evans function which includes spectral stability---is known to imply nonlinear stablity for viscous shock profiles in a variety of hyperbolic-parabolic systems; see, e.g., \cites{MaZ1,MaZ2,Z1,Z2}. 


In this paper, we focus on a practical issue that arises in the computation of $D(\lambda)$ for physical systems like the Navier--Stokes equations (equations \eqref{eq:ns} or equations \eqref{eq:Lag}). The main message is a cautionary tale in that a natural coordinate system may not be the 
``best'' one. That is, while Eulerian coordinates are often used in the computational fluid dynamics community (for direct numerical simulations of the flow), \emph{we find that these coordinates lead to an Evans function that is practically incomputable} for intermediate frequencies and moderate shock strengths. In particular, we find that the output of the Eulerian Evans function varies dramatically, both in modulus and argument. Since stability calculations are usually done by winding number counts on the image of a semi-annular contour in the unstable complex half plane, rapid changes in modulus and argument lead to computations that are prohibitively complicated and expensive. In particular, this leaves physical models with many parameters and virtually any multidimensional problem out of reach. Thus, despite the existence of mature packages, i.e., \textsc{STABLAB}, for Evans-function computations, one cannot simply feed a coefficient matrix $\mat{A}$ into a package and ``hope for the best.''  

\subsection{Multidimensional formulation}\label{ssec:multid}
The Eulerian-coordinates-based obstacle 
is present in both one and several spatial dimensions. However, in a single space dimension, the issue can easily be sidestepped by working with the Lagrangian form of the equations. In multiple space dimensions, however, this maneuver is not available, and one must confront the issue head on.
Thus, although the main analysis of this paper takes place in a single space dimension, we now outline the general set-up for the multidimensional case as a preliminary to the calculations in \S\ref{sec:pseudo} where we illustrate the effectiveness of our pseudo-Lagrangian coordinates for two-dimensional isentropic gas dynamics. Indeed, we expect that our findings will be critical for Evans-based analysis of problems in multidimensional magnetohydrodynamics and detonation theory. 

Generalizing equation \eqref{eq:claw}, consider now a system of $n$ conservation laws with 
viscosity in $d$ space dimensions:
\beq\label{eq:claw2}
f^0(U)_t+\sum_{j=1}^d f^j(U)_{x_j}=\sum_{j,k=1}^d (B^{jk}(U)U_{x_k})_{x_j}\,.
\eeq
In equation \eqref{eq:claw2}, $x=(x_1,\ldots,x_d)\in\RR^d$, $t\in\RR$, and $U\in\RR^n$ with
\[
f^j:\RR^n\to\RR^n\,,j=0,1,\ldots,d\,;
\quad
B^{jk}:\RR^n\to\RR^{n\times n}\,,j,k=1,\ldots,d\,.
\]
We write $A^j(U)\defequals\dif f^j(U)$ for $j=0,1,\ldots,d$.

As above, our interest is in the stability of planar viscous shock profiles. 
Thus, we consider traveling-wave solutions of the form
\beq
U(x,t)=\bar U(x_1-st)\,,\; \lim_{z\to\pm\infty} \bar U(z)=U_\pm \,,
\eeq
and, without loss of generality, we assume $s=0$. 
Similarly as above, we
linearize about the steady solution $\bar U$ to obtain a linear equation for a small perturbation $V=V(x,t)$. That equation is
\beq\label{eq:meval}
A^0(x_1)V_t+\sum_{j=1}^d(A^j(x_1) V)_{x_j}=\sum_{j,k=1}^d(B^{jk}(x_1) V_{x_k})_{x_j}\,,
\eeq
where
\begin{align*}
A^0(x_1)&\defequals A^0(\bar U(x_1))\,,
\\
A^j(x_1)V&\defequals A^j(\bar U(x_1))V-\dif B^{j1}(\bar U(x_1))\big(V,\bar U'(x_1)\big)\,,
\\
B^{jk}(x_1)&\defequals B^{jk}(\bar U(x_1))\,.
\end{align*}
We take the Laplace transform in time (dual variable $\lambda$) and Fourier transform (dual variable $\xi=(\xi_2,\ldots,\xi_d)$) in the transverse spatial directions $(x_2,\dots,x_d)$, and we find
the generalized eigenvalue equation (supressing the dependence of the coefficients on $x_1$)
\begin{multline}\label{eq:mgeval}
\lambda A^0W+(A^1 W)' + \sum_{j=2}^d\mi\xi_j A^j W
= (B^{11} W')' + \sum_{k= 2}^d (\mi\xi_k B^{1k} W)' \\
+\sum_{j=2}^d\mi\xi_j B^{j1}W' - \sum_{j, k=2}^d\xi_j\xi_k  B^{jk} W\,.
\end{multline}
In equation \eqref{eq:mgeval}, $W=W(x_1,\lam, \xi)$ represents the transformed perturbation. 
As above, we reformulate the eigenvalue problem \eqref{eq:mgeval} as a first-order system of differential equations
\beq\label{eq:1order}
Z'=\mat{A}(x_1;\lam,\xi)Z\,.
\eeq
Here, $\mat{A}$ is an $N\times N$ matrix where the dimension $N$ depends on the structure of the system \eqref{eq:claw2}\footnote{We have omitted any mention of structural hypotheses on the system \eqref{eq:claw2}.}, and since $\bar U$ decays rapidly to its limiting values $U_\spm$ as $x_1\to\pm\infty$, then the coefficient matrix $\mat{A}$ also has constant (with respect to $x_1$) limiting values. We denote these by $\mat{A}_\spm(\lam,\xi)$. 

\begin{remark}[Flux Coordinates]
 A systematic way to formulate the first-order system \eqref{eq:1order} is to use one of the variations of flux coordinates \cite{BHLZ2}. 
 These coordinates confer concrete benefits for the numerical approximation of the Evans function, and are especially useful for multidimensional problems \cite{HLyZ3}.
 \end{remark}

Then, as above, the Evans function is built out of the subspaces of solutions of equation \eqref{eq:1order} which grow at $-\infty$ and decay at $+\infty$; the construction of these subspaces starts with an analysis of the constant-coefficient limiting system $Z'=\mat{A}_\spm(\lam,\xi)Z$.
That is, if the collection
$\{z_1^\sp,\ldots,z_k^\sp\}$ forms a basis for the solutions of equation \eqref{eq:1order} that decay at $+\infty$ and, similarly, $\{z_{k+1}^\sm,\ldots z_{N}^\sm\}$ spans the solutions that grow at $-\infty$, the Evans function can be written as 
\beq\label{eq:abstractevans}
D(\lam, \xi)\defequals\det(z_1^\sp,\ldots,z_k^\sp,z_{k+1}^\sm,\ldots z_{N}^\sm)|_{x_1=0}\,.
\eeq
Thus, $D$ is a function of frequencies,
\[
D:\{\lam\in\CC\,:\,\Re \lam>0\}\times \RR^{d-1}\to\CC\,,
\]
whose zeros correspond to eigenvalues, and the principal goal is to compute $D$ (or its zero set). 

\subsection{Outline}\label{ssec:outline}
In \S\ref{sec:coords} we recall the fundamentals of the Eulerian and Lagrangian coordinate systems for gas-dynamical models. For simplicity and concreteness, we carry out these calculations in one space dimension and in the setting of isentropic gas dynamics. Next, in \S\ref{sec:evans} we describe the two Evans functions arising from the pair of coordinate systems and illustrate their performance, again in the setting of one-dimensional isentropic gas dynamics. In \S\ref{sec:explanation} we describe the mathematical origin of the observed discrepancy in behavior between the Eulerian and Lagrangian Evans functions. We turn to the multidimensional case in \S\ref{sec:pseudo}, and we introduce there a ``pseudo-Lagrangian'' Evans function. This Evans function is based on Eulerian coordinates but shares the favorable properties of the one-dimensional Lagrangian Evans function. We illustrate the performance of this new Evans function by considering planar viscous shocks in two-dimensional isentropic gas dynamics.
Finally, in \S\ref{sec:conclusions}, we collect our findings and discuss their implications.

\section{Eulerian vs. Lagrangian coordinates}\label{sec:coords}

We recall that in continuum physics there are two distinct ways to describe the motion of a fluid. The Eulerian description assigns values to points in the physical domain;
thus, $\rho(x,t)$ is the density of the fluid particle that occupies the location $x$ at the instant $t$. 
The Lagrangian description is based on an initial labeling of all the fluid particles at some initial instant and tracking them as the fluid moves. 
Thus, $\tau(y,t)\defequals\rho(y,t)^{-1}$ represents the specific volume at the instant $t$ of the fluid particle 
marked with the label $y$. 
We begin by reviewing the Eulerian and Lagrangian descriptions of viscous shocks.

\subsection{Eulerian coordinates}
The one-dimensional isentropic Navier--Stokes equations in Eulerian coordinates are
\begin{subequations}\label{eq:ns}
\beq
\rho_t+(\rho u)_x = 0\,,
\eeq
\beq
(\rho u)_t+(\rho u^2+p(\rho))_x=u_{xx}\,,
\eeq
\end{subequations}
where we have, without loss of generality, set the coefficient of viscosity to be $1$. 
For definiteness, we assume a polytropic, or ``$\gamma$-law,'' pressure law
\begin{equation}\label{eq:pressure}
p(\rho)=a\rho^\gamma, \qquad a, \gamma >0\,.
\end{equation}
This is not important for our main conclusions, but this assumption streamlines and simplifies the surrounding discussion.

As noted above, a viscous shock is an asymptotically constant traveling-wave solution of equation \eqref{eq:ns}. That  is, it is a solution of form
\[
\rho(x,t)=\bar \rho (x-\sigma t)\,,
\quad
u(x,t)=\bar u(x-\sigma t)
\]
connecting constant states $(\rho_\spm,u_\spm)$. That is, the viscous shock satisfies
\[
\lim_{z\to\pm\infty} (\bar\rho (z), \bar u (z)) = (\rho_\spm, u_\spm)\,.
\]
Due to Galilean invariance, without loss of generality, we may assume that the traveling wave of interest is stationary. That is, the wave speed $\sigma$ is zero. 
This reduces the traveling-wave equation to the the time-independent part of equation \eqref{eq:ns}, namely (dropping bars and using prime to denote differentiation with respect to $x$)
\begin{equation}\label{eq:nstw}
(\rho u)' = 0\,,\quad (\rho u^2+p(\rho))'=u''\,. 
\end{equation}
Integrating equation \eqref{eq:nstw} from $-\infty$ to $+\infty$, we obtain the Rankine-Hugoniot jump conditions
\begin{align}\label{erh}
[\rho u] & = 0\,, \qquad
[\rho u^2+p(\rho)]=0\,,
\end{align}
where $[\cdot]$ denotes difference between limits at $+\infty$ and $-\infty$.
It is straightforward to verify that, for a $\gamma$-law gas,
for each pair of endstates $(\rho_\spm, u_\spm)$ obeying equation \eqref{erh}, there exists a
unique heteroclinic connection corresponding to a traveling wave.
More, for each choice of momentum flux $m\defequals\rho_\spm u_\spm$, it can be seen that there is a unique solution of equation \eqref{erh}, hence
a unique associated stationary shock.

\subsection{Lagrangian coordinates}
To convert to Lagrangian coordinates, we set
$$
y(x,t)=\int_{x^*(t)}^x \rho(z,t)\,\dif z
$$
with 
$x^*(0)=0$, $\frac{\dif x^*}{\dif t}=u(x^*(t),t)$.
Then, we observe that 
\beq\label{eq:dydx}
\frac{\pd y}{\pd x}(x,t)=\rho(x,t)\,,
\eeq
and
\begin{align}
\frac{\pd y}{\pd t}(x,t) & = \int_{x^*(t)}^x \frac{\pd \rho}{\pd t}(z,t)\,\dif z -\rho(x^*(t),t)\frac{\dif x^*}{\dif t} \nonumber\\
	&=-\int_{x^*(t)}^x\pd_z(\rho u)\,\dif z -\rho(x^*(t),t)u(x^*(t),t) \nonumber \\
	&=-\rho(x,t)u(x,t)+\rho(x^*(t),t)u(x^*(t),t)-\rho(x^*(t),t)u(x^*(t),t) \nonumber\\
	&=-\rho u(x,t)\,.\label{eq:dydt}
\end{align}
Thus, defining
\beq
\tau(y(x,t),t)=\frac{1}{\rho(x,t)}\,,
\quad
w(y(x,t),t)=u(x,t),\
\eeq
and denoting by $P$ the pressure as a function of specific volume, we find---using equations \eqref{eq:dydx} and \eqref{eq:dydt}---that 
the Lagrangian formulation of system \eqref{eq:ns} is
\begin{subequations}\label{eq:Lag}
\beq
\tau_t  -w_y=0\,,
\eeq
\beq
w_t + P(\tau)_y= \Big(\frac{w_y}{\tau}\Big)_y.
\eeq
\end{subequations}

\br\label{indrmk}
Note that this change of coordinates involves both dependent and independent variables;
see, e.g., Courant \& Friedrichs \cite{CF} or Serre \cite{S} for further details.
\er

From equation \eqref{eq:Lag}, the traveling-wave equation for a traveling-wave solution of form 
\[
\tau(y,t) =\bar\tau(y-st)\,, 
\quad
w(x,t)= \bar w(y-st)
\]
with 
$\lim_{\zeta\to\pm\infty}(\bar\tau(\zeta),\bar w(\zeta))=(\tau_\spm,w_\spm)$
is thus
\begin{equation}\label{eq:lagtw}
	-s\tau' = w'\,,\quad -sw'+ P(\tau)'= (w'/\tau)'\,. 
\end{equation}
Here, $'$ denotes differentiation with respect to $\zeta\defequals y-st$. 
Integrating from $-\infty$ to $+\infty$, we obtain the Lagrangian version of the Rankine-Hugoniot conditions
\eqref{erh}:
\begin{align}\label{lrh}
-s[\tau]-[w] & =0 \,, 
\quad 
-s[w]+[P(\tau)]  =0 \,.
\end{align}
Using 
\(
\rho_\sp u_\sp=\rho_\sm u_\sm =m,
\)
we may rewrite the jump condition as 
\beq
m[u]=-[p]=-[P]=-s[w]\,,
\eeq
whence $m=-s$.  This relation is useful in comparing Eulerian versus Lagrangian shock parametrizations 
without appealing to the full coordinate transformation.

\section{Evans functions and their performance}\label{sec:evans}
We now construct the Evans function in Eulerian and Lagrangian coordinates following
\cite{BHLZ2}, and we compare their respective performances.
Using the invariances of $\gamma$-law gas dynamics \cite{HLZ},
we take without loss of generality $m=-s=1$ and $\rho_\sm=1$ in what follows, parametrizing the strength
of the shock by $u_\sp$ ($\tau_\sp$) in the Eulerian (Lagrangian) case, where---as above---$\pm$ subscripts denote limits
at $\pm\infty$ of corresponding coordinates.

\subsection{Eulerian case}
Linearizing equation \eqref{eq:ns} about a steady profile $(\bar \rho, \bar u)$, we obtain the eigenvalue problem
\begin{subequations}\label{evaleuler}
\beq
\lambda \rho + (\bar \rho u+\rho \bar u)'= 0\,,
\eeq
\beq
\lambda(\bar \rho u + \rho \bar u) + (\rho \bar u^2+2u+p'(\bar \rho) \rho)' = u''\,.
\eeq
\end{subequations}
Defining 
\begin{equation}
\beta\defequals \frac{\bar u^2 + p'(\bar \rho)}{\bar u}\quad \textrm{and} \ 
f\defequals \begin{pmatrix} -\rho \bar u-\bar \rho u \\ u'-2u-\beta\bar u \rho \end{pmatrix}
\notag
\end{equation}
we may rewrite the eigenvalue problem as the first order system
\begin{equation}
\begin{pmatrix} f\\ u\end{pmatrix}' = 
\begin{pmatrix} -\lambda /\bar u &0 & -\lambda \bar \rho/\bar u \\ -\lambda & 0 & 0\\ -\beta & 1 & 2-\beta \bar \rho \end{pmatrix}\begin{pmatrix} f\\u\end{pmatrix},
\label{eq:evans_ode_eulerian}
\end{equation}
or, briefly, 
\be\label{eW}
\frac{\dif}{\dif x} W=\mat{A}(x;\lambda)W\,,
\quad 
\text{where}\quad 
W=\bp f,u\ep\,.
\ee

Eigenvalues of equation \eqref{evaleuler} correspond to values of $\lambda$ for which there exist solutions of
equation \eqref{eW} decaying as $x\to \pm \infty$, that is, a nontrivial intersection of the manifolds of solutions
decaying at $\pm \infty$.  By standard asymptotic results from ordinary differential equations (ODEs)---e.g., the ``gap Lemma'' of \cite{GZ})---one finds
that these manifolds are spanned by bases $\{W_1,W_2\}$ and $\{W_3\}$
asymptotic to eigenmodes $\me^{\mu_j x}V_j$ of the stable (unstable) subspaces of the limiting coefficient matrices 
$\mat{A}_\spm\defequals \mat{A}(\pm \infty;\lambda)$, where $\mu_j, V_j$ depend on $\lambda$.
The Evans function associated with equation \eqref{evaleuler} is then defined as
\begin{equation}\label{evanseuler}
	D_\mathrm{E}(\lambda)\defequals\det (W_1,W_2,W_3)|_{x=0}.
\end{equation}
Here, an important detail is the specification of the ``initializing bases at $\pm \infty$'' 
$V_j$; these are defined as solutions of Kato's ODE \cite{Kato}
\be\label{Katode}
\dif R/\dif \lambda =(\dif\PP/\dif \lambda) R,
\ee
where $\PP(\lambda)$ is the (uniquely determined)
projection onto the stable (unstable) subspace of $\mat{A}_\pm(\lambda)$, and $R$ is a matrix whose columns form
the bases $V_j$. 

This determines the Evans function uniquely up to a constant factor, which is then normalized by 
setting $D_\mathrm{E}(\lambda_*)=1$ at some convenient initializing frequency $\lambda_*$ (typically 
the maximum real value of frequencies under consideration).
It may be checked that the above definition makes sense, i.e., the counts of stable/unstable basis elements
are correct, on the unstable region $\Re \lambda\geq 0$, $\lambda\neq 0$, 
where dimensions of stable/unstable subspaces of $\mat{A}_\spm$ agree.

\subsection{Lagrangian case}
The eigenvalue equation in Lagrangian coordinates is
\begin{equation}
	\begin{split}
		\lambda \tau + \tau'-u'&= 0\\
		\lambda u + u' -(P'(\bar \tau) \tau)'&= \left( \frac{u'}{\bar \tau}-\frac{\bar u ' \tau}{\bar \tau^2}\right)',
	\end{split}
\end{equation}
where $\mathsf P'(\bar \tau) = a\gamma \bar \tau^{-\gamma-1} $. This may evidently be written as the first order system 
\begin{equation}
\begin{pmatrix} \tau\\ u\\ u' \end{pmatrix}'= \begin{pmatrix} -\lambda & 0 & 1\\ 0 & 0 & 1\\ \lambda \alpha \bar \tau  & \lambda \bar \tau & \bar \tau(1-\alpha) \end{pmatrix} \begin{pmatrix} \tau\\ u\\ u' \end{pmatrix},
\label{eq:evans_ode_lagrange}
\end{equation}
or 
\be\label{eL}
\frac{\dif }{\dif y}Z=\mat{B}(y;\lambda)Z\,,
\quad
\text{where}
\quad
Z=\bp \tau& u & u'\ep\,,
\ee
where $\alpha\defequals P'(\bar \tau) - \frac{\bar u'}{\bar \tau^2}$; equivalently one may follow the more complicated, but in this case unnecessary, prescription of \cite{BHLZ2}.
The Lagrangian Evans function $D_\mathrm{L}(\lambda)$ is then defined, similarly as in the Eulerian case, as
\begin{equation}\label{evanslag}
	D_\mathrm{L}(\lambda)\defequals\det (Z_1,Z_2,Z_3)|_{y=0},
\end{equation}
where the stable (unstable) manifolds of the flow of equation \eqref{eL} at $+\infty$ ($-\infty$)
are spanned by bases $\{Z_1,Z_2\}$ and $\{Z_3\}$
asymptotic to eigenmodes $\me^{\nu_j y}U_j$ of the stable (unstable) subspaces of the limiting coefficient matrices 
$\mat{B}_\pm\defequals\mat{B}(\pm \infty;\lambda)$, with $U_j$ prescribed via Kato's ODE
\be\label{Katode2}
\dif S/\dif \lambda =(\dif \QP/\dif \lambda) S,
\ee
where $\QP(\lambda)$ is the (uniquely determined)
projection onto the stable (unstable) subspace of $\mat{B}_\spm(\lambda)$, and $S$ is a matrix whose columns form
the bases $U_j$. 
Again, the above prescription is well-defined on the unstable region $\Re \lambda \geq 0$, $\lambda \neq 0$.

%

\subsection{Numerical performance}\label{ssec:numerical-performance}
Despite the apparent similarity of Evans functions $D_\mathrm{E}$ and $D_\mathrm{L}$, their performance is quite different for practical 
computations.  These computations typically consist of winding number computations on the image under the Evans function of a semi-annular
contour determined (by energy estimates or auxiliary asymptotic ODE estimates) to contain all possible unstable eigenvalues
of the linearized operator about the wave.
A winding number of zero thus corresponds to spectral stability while a nonzero winding number signals the presence of unstable eigenvalues and therefore instability.

In Figure \ref{fig:eulerian_evans1}, we plot a
traveling-wave solution of equation \eqref{eq:ns} and the Evans function, evaluated on a semi-annulus (see Figure \ref{fig:compare}(b)) with inner radius $r=10^{-3}$ and outer radius $R =(1/2+\sqrt{\gamma})^2$, as computed with the Eulerian coordinates formulation given in equations 
\eqref{eq:evans_ode_eulerian}, \eqref{eW}.  The Evans function maps contours of the form shown in Figure \ref{fig:compare}(b) to contours of the form shown in Figure \ref{fig:compare}(c). To compute the Evans function, we use the method of continuous orthogonalization described in \cite{HuZ2}. All computations are carried out in \textsc{STABLAB} \cite{STABLAB}.  We note that in Eulerian coordinates, the Evans function contour wraps around the origin 10 times before unwrapping to yield winding number zero. Further, the Evans function varies over 12 orders of magnitude (from 1 to approximately 2.8e12). This is in stark contrast to the Evans function in Lagrangian coordinates, which is bounded away from the origin and remains order one in modulus (varying from 1 to about 0.2). In Figure \ref{fig:compare}(a)--(d), we plot the profile solution to equation \eqref{eq:Lag} and the Evans function, evaluated on a semi-annulus with inner radius $r=10^{-3}$ and outer radius $R =(1/2+\sqrt{\gamma})^2$, as computed with the Lagrangian coordinates formulation given in equation \eqref{eq:evans_ode_lagrange}. One can see by mere observation that the contour featured in Figures \ref{fig:compare} (c) and (d) has a winding number of zero.

\begin{figure}[ht!]
 \begin{center}
$
\begin{array}{lr}
\text{(a)} \includegraphics[scale=0.3]{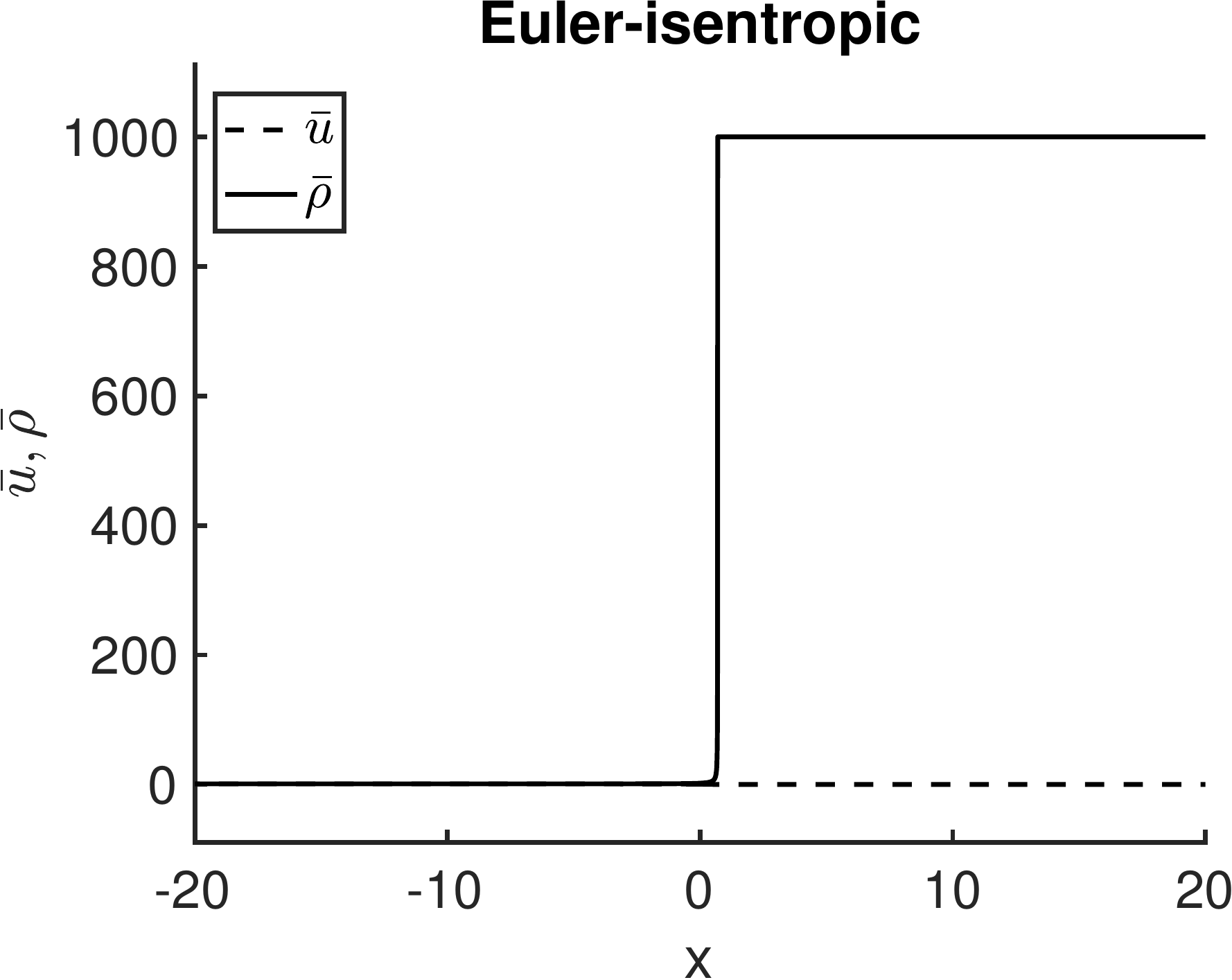} & \text{(b)} \includegraphics[scale=0.3]{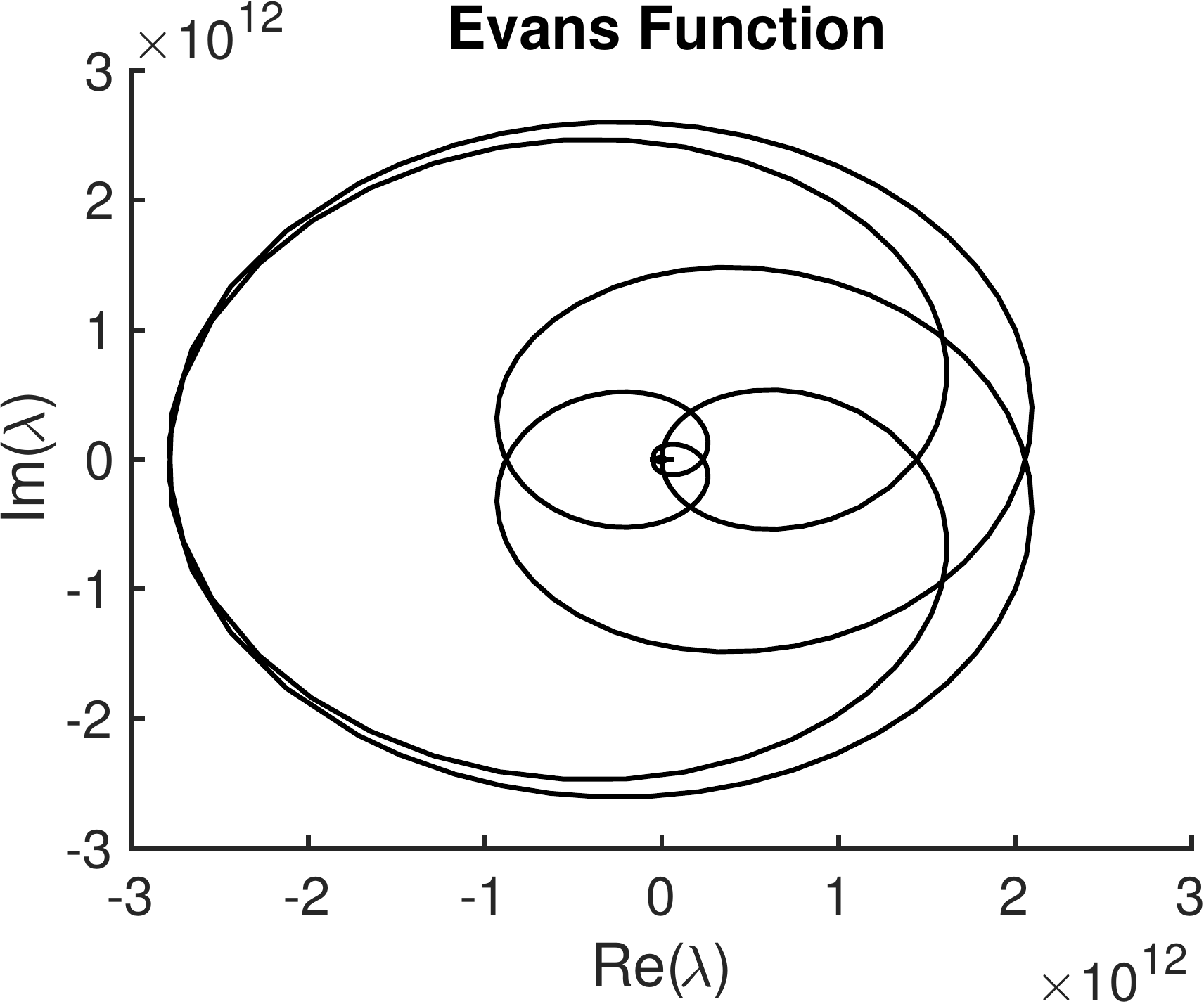} \\
\text{(c)} \includegraphics[scale=0.3]{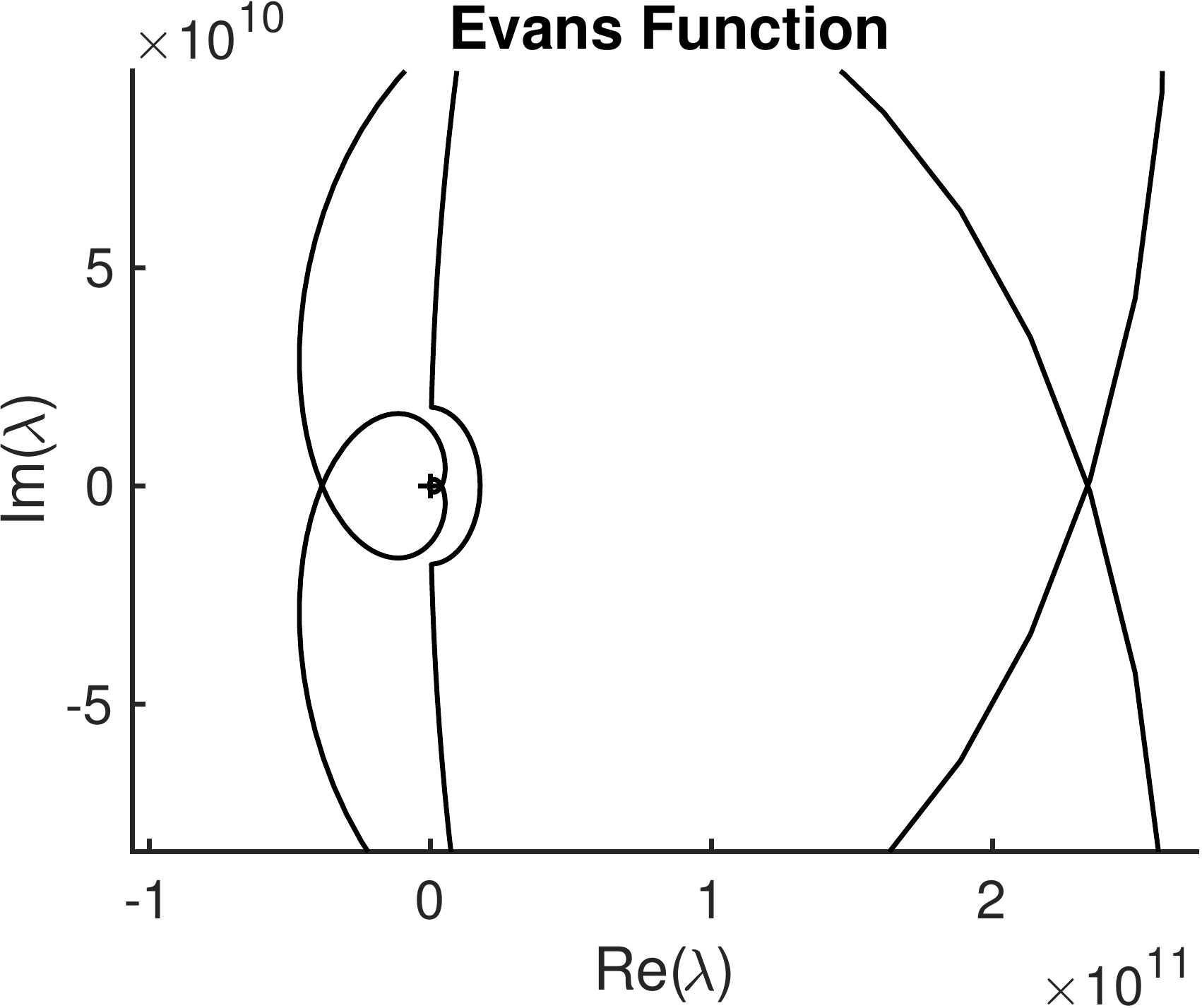} & \text{(d)} \includegraphics[scale=0.3]{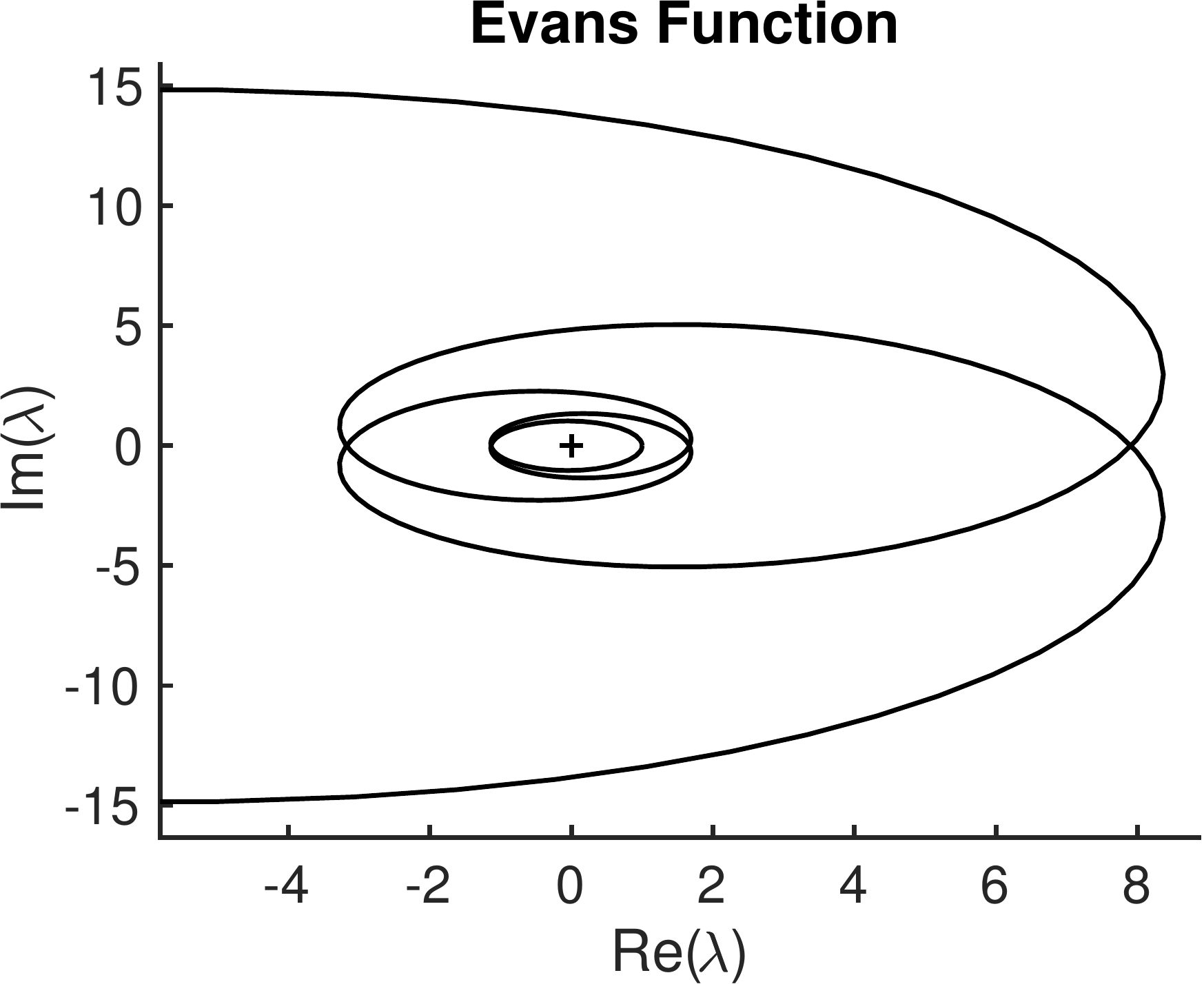}
\end{array}
$
\end{center}
\caption{Plot of the profile and Evans function for one-dimensional isentropic gas in Eulerian coordinates when $\gamma = 5/3$ and $u_\sp = 0.001$. (a) Traveling wave profile. (b) Evans function evaluated on a semi-annulus contour with inner radius $r=10^{-3}$ and outer radius $R =(1/2+\sqrt{\gamma})^2$. (c) Zoom in of Figure (b). (d) Zoom in of Figure (c). Throughout a + marks the origin.}
\label{fig:eulerian_evans1}
\end{figure}

\begin{figure}[ht!]
 \begin{center}
$
\begin{array}{lcr}
\text{(a)} \includegraphics[scale=0.3]{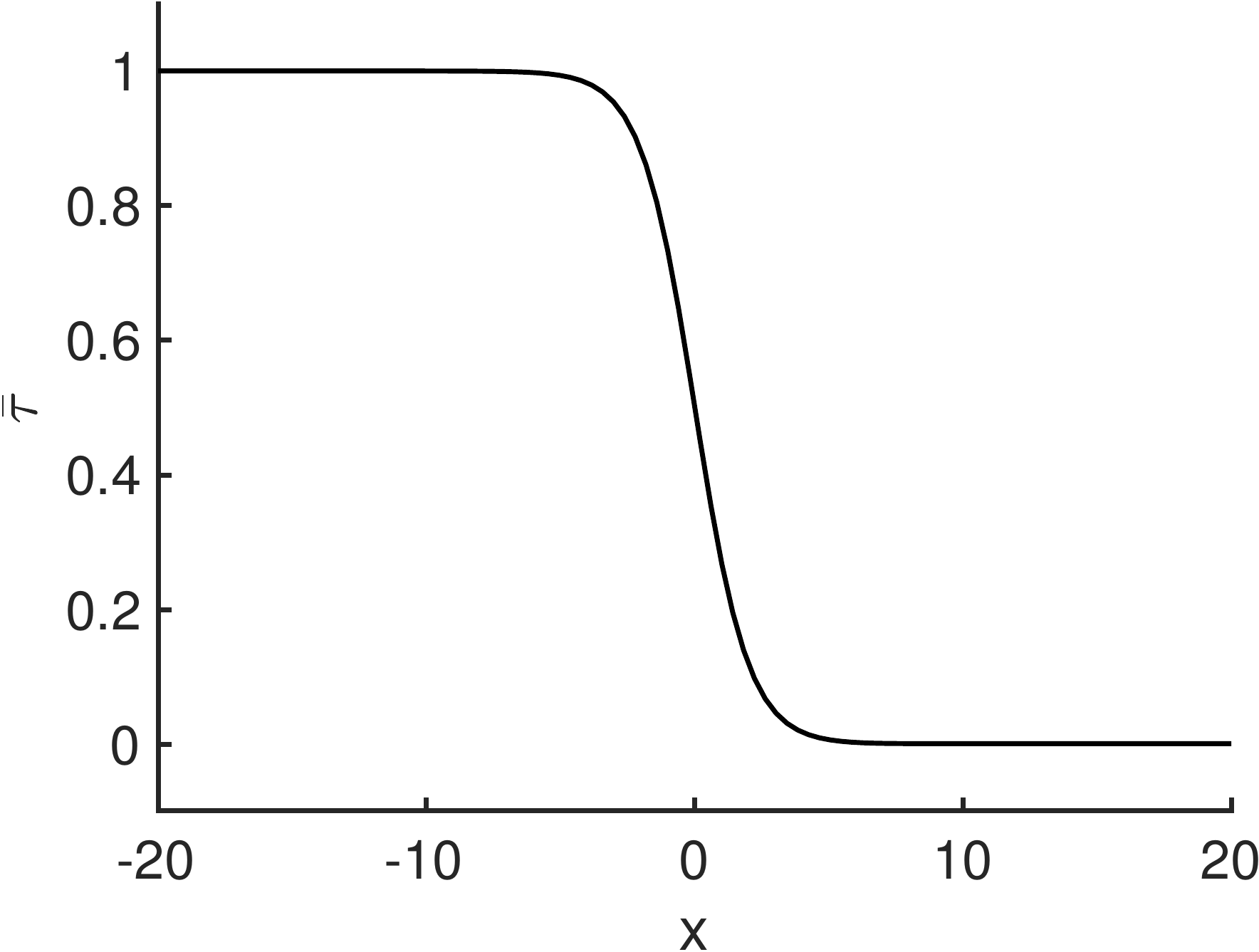}& \text{(b)} \includegraphics[scale=0.3]{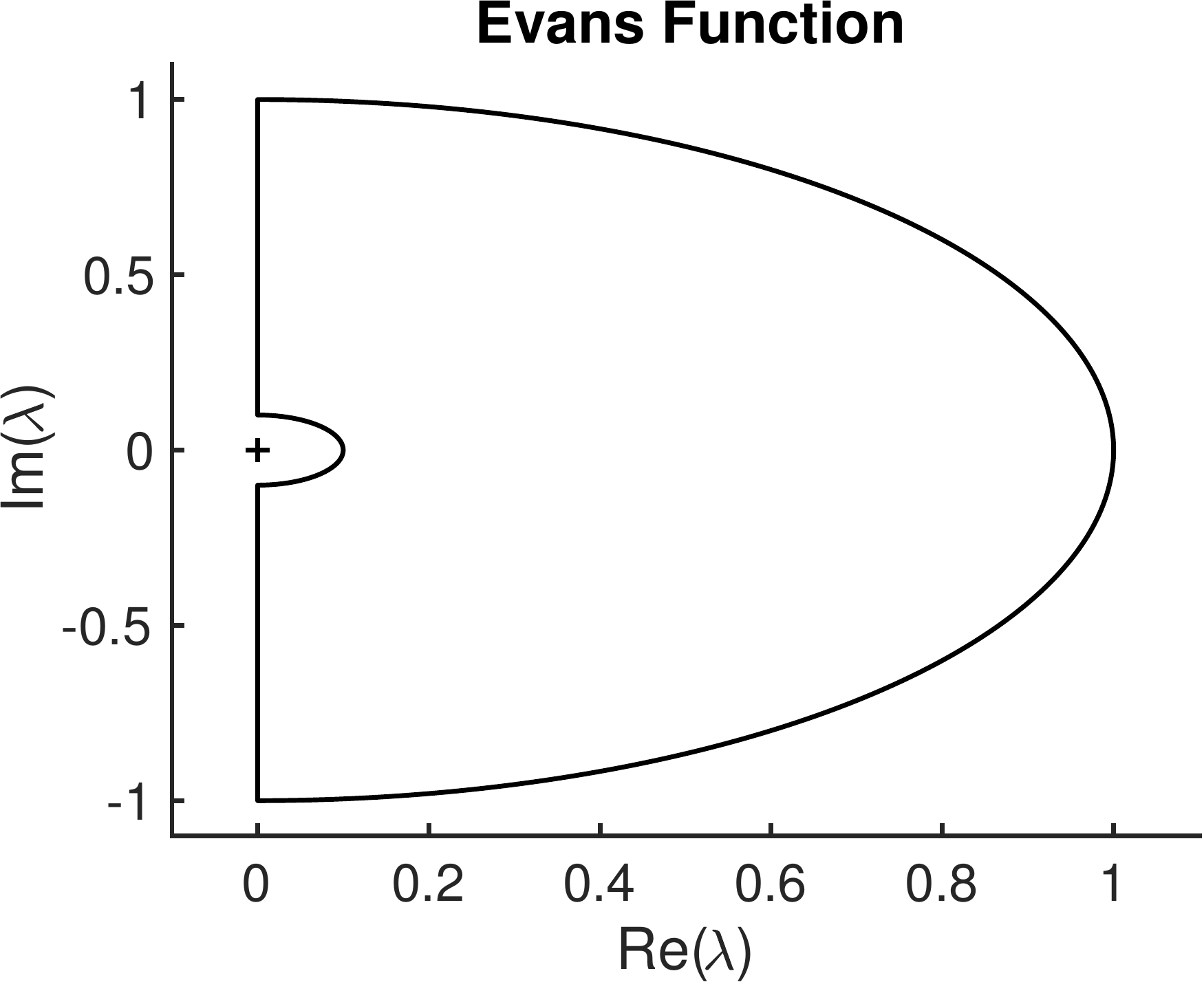} \\ 
\text{(c)}  \includegraphics[scale=0.3]{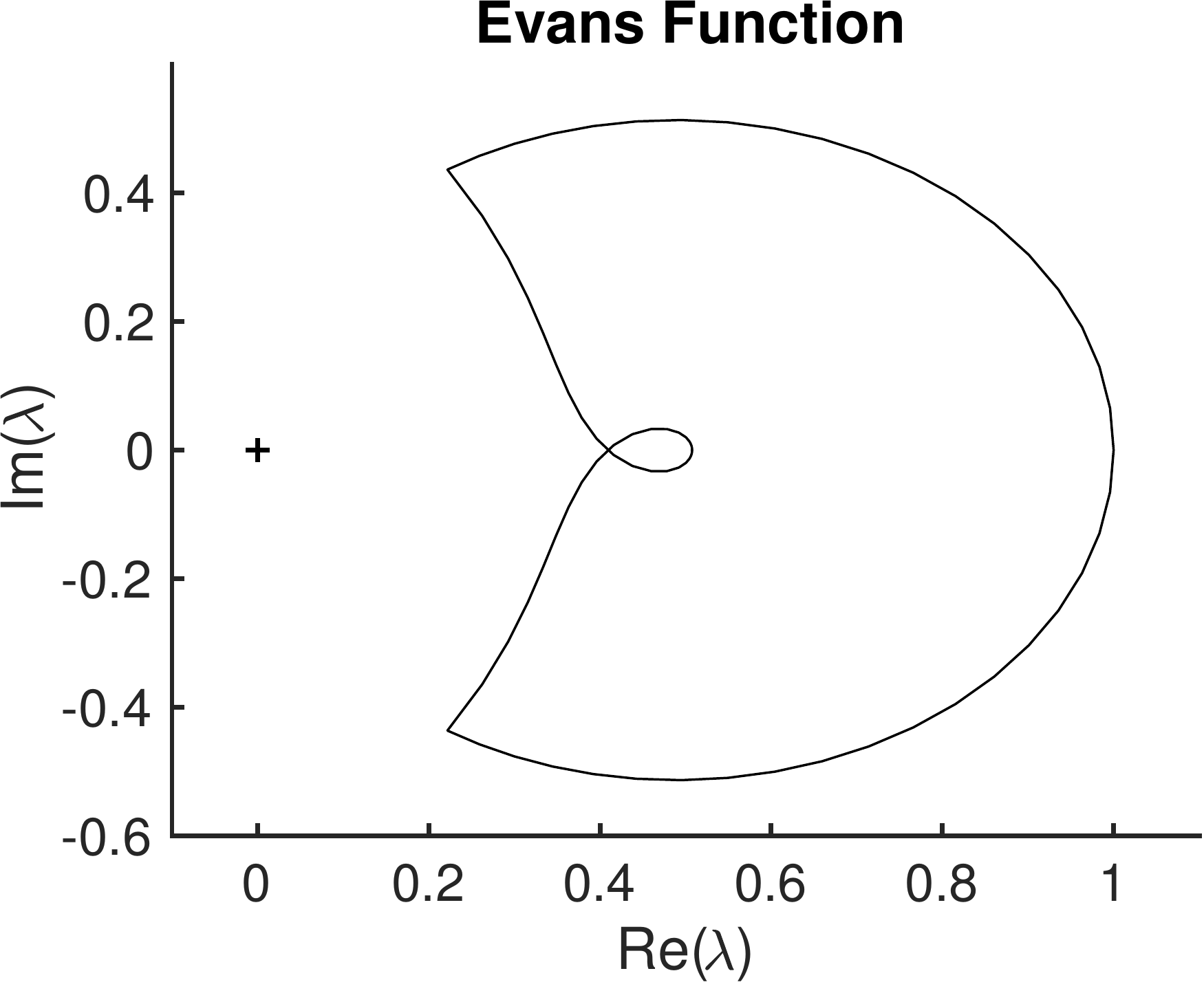} & \text{(d)} \includegraphics[scale=0.3]{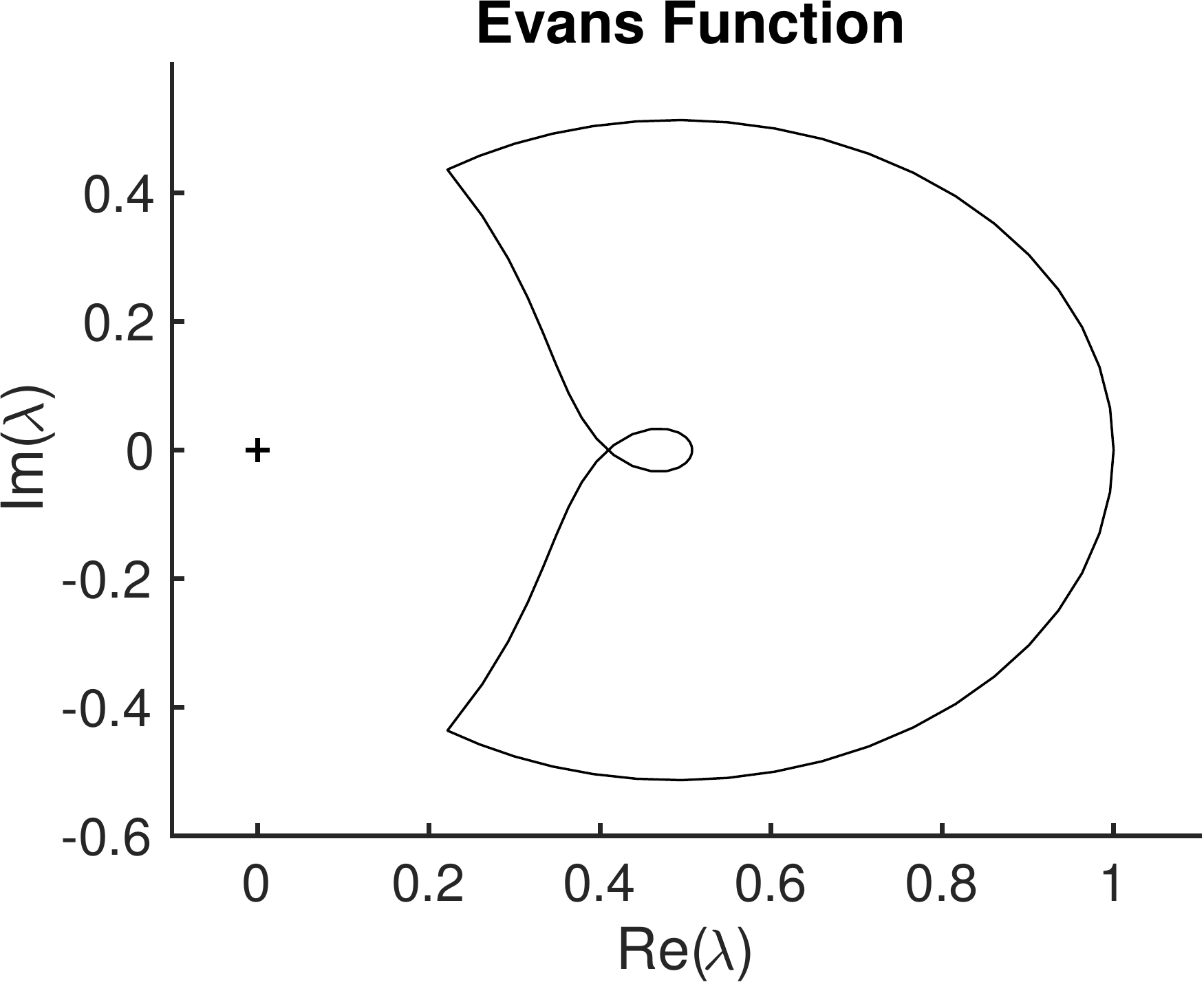}
\end{array}
$
\end{center}
\caption{Plot of the profile and Evans function for one-dimensional isentropic gas in Lagrangian and pseudo-Lagrangian coordinates when $\gamma = 5/3$ and $u_\sp = 0.001$. (a) Profile in Lagrangian coordinates. (b) Example of the type of semi-annulus contour on which we compute the Evans function. (c) Evans function in Lagrangian coordinates. (d) Evans function in pseudo-Lagrangian coordinates. Throughout, a + marks the origin. }
\label{fig:compare}
\end{figure}

To summarize, in comparison with the Lagrangian Evans function, the Eulerian Evans function exhibits excessive winding. This makes spectral computations prohibitively complicated and expensive in the Eulerian formulation (as noted earlier, a serious problem in the multidimensional case).

\section{Explanation of observed results}\label{sec:explanation}
We now investigate the origins of the discrepancy between the Eulerian and Lagrangian Evans functions.
Evidently, the flows of the Evans systems \eqref{eq:evans_ode_eulerian} and 
\eqref{eq:evans_ode_lagrange} are conjugate, hence, noticing that we have normalized so that $y(0)=0$, 
{\it up to the initialization at $\pm \infty$}, we observe that the two Evans functions should agree up to a nonzero constant factor
equal to the determinant at $x=y=0$ of the $\lambda$-independent coordinate transformation between 
$\begin{pmatrix} f,u\end{pmatrix} $ and $\begin{pmatrix} \tau,u,u'\end{pmatrix} $.
	Thus, the discrepancy can only originate from two sources: 
\begin{itemize}
\item[(i)] the prescription of $V_j(\lambda)$ via
	Kato's ODE, or 
\item[(ii)] the asymptotic prescription $W_j\sim \me^{\mu_j(\lambda)}V_j(\lambda)$ as $x\to \pm \infty$.
\end{itemize}

\subsection{The conjugating transformations}
The relationship between dependent coordinates is given, linearizing the relation $\rho=\tau^{-1}$, by
the pair of transformations
$$
\begin{pmatrix}
\rho \\
u \\
u'
\end{pmatrix}
=
\begin{pmatrix}
-\frac{1}{\bar\tau^2} & 0 & 0 \\
0 & 1 & 0 \\
0 & 0 & 1
\end{pmatrix}
\begin{pmatrix}
\tau \\
w \\
w'
\end{pmatrix}
$$
and
$$
\begin{pmatrix}
f \\
u
\end{pmatrix}
\defequals
\begin{pmatrix} -\bar u \rho -\bar \rho u \\ u'-2u-\beta\bar u \rho\\u \end{pmatrix}
=
\begin{pmatrix}
-\bar u & -\bar \rho & 0 \\
-\beta \bar u & -2 & 1 \\
0 & 1 & 0
\end{pmatrix}
\begin{pmatrix}
\rho \\
u \\
u'
\end{pmatrix},
$$
the composition of which gives a {\it $\lambda$-independent conjuagator $T(x)$} such that 
\be\label{depcon}
\begin{pmatrix} f \\ u \end{pmatrix}=T(x) \begin{pmatrix} \tau \\ w \\ w' \end{pmatrix}.
	\ee
	The relation between independent variables is likewise \emph{$\lambda$-independent}, given 
	(see \eqref{eq:dydx}) by
	\be\label{yx}
	\dif y/\dif x=\bar\rho(x).
	\ee
	Combining these two observations, the relation between equations \eqref{eW} and \eqref{eL} is thus
	\be\label{AB}
	\mat{B}(y(x);\lambda)= \bar \rho(x)^{-1} T(x)^{-1}\mat{A}(x;\lambda)T(x).
	\ee

\subsection{Invariance of Kato's equation}
Having observed the $\lambda$-independence of the conjugating transformations, we may eliminate the 
possibility (i) as a source of discrepancy between the two Evans function formulations based on the
following general result.

\bl\label{katolem}
Kato's ODE is invariant under $\lambda$-independent coordinate changes.
\el

\begin{proof}
	Focusing on either $x=+\infty$ or $x=-\infty$, it is sufficient by equation \eqref{AB} 
	to consider constant coefficient matrices
$\mat{A}$, $\mat{B}$, related by $\mat{B}=\rho^{-1} T\mat{A}T^{-1}$, $\QP=T\PP T^{-1}$, where
$\PP$ and $\QP$ are projections onto the stable (unstable) subspaces of $\mat{A}$ and $\mat{B}$,
with $\rho\in \R$ and $T\in \R^{3\times 3}$
constant.
Then, the Kato ODEs for the two systems are 
\be\label{P}
\dif R/\dif \lambda=(\dif \PP/\dif\lambda) R
\ee
and
\be\label{Q}
\dif S/\dif \lambda=(\dif\QP/\dif\lambda)S,
\ee
and the claim is that $S\defequals TR$ is a solution of equation \eqref{Q}
if and only if $R$ is a solution of equation \eqref{P}.
Computing
$ S'=TR'=T\PP'R=T\PP'T^{-1}S=\QP'S$, we are done.
\end{proof}

\subsection{Asymptotic prescription of basis elements}
Having eliminated possibility (i), we now explicitly relate the Eulerian and Lagrangian Evans functions by examination
of (ii).
On the unstable region $\Re \lambda \geq 0$, $\lambda\neq 0$ where our prescriptions of the Evans functions
are well-defined, let $\nu_\sp$ denote the sum of the stable (negative real part) 
eigenvalues of $\mat{A}_\sp$ and $\nu_\sm$ the sum of the unstable (positive real part) eigenvalues of $\mat{A}_\sm$.
Define constants
\be\label{Deltas} 
\Delta_\sp\defequals \int_0^{+\infty} (\bar\rho(x)-\rho_\sp)\,\dif x
\quad
\text{and}
\quad
\Delta_\sm\defequals\int_{-\infty}^0 (\bar\rho(x)-\rho_\sm)\,\dif x\,.
\ee

\begin{lemma}\label{evrel}
For $T$ as in equation \eqref{depcon}, 
the Eulerian and Lagrangian Evans functions are related by 
\be\label{phi}
D_\mathrm{E}(\lambda)=
\det T(0) \me^{\nu_\sp\Delta_\sp -\nu_\sm\Delta_\sm}D_\mathrm{L}(\lambda),
\ee
where, for $m=-s=1$,  
\[
\nu_\sp\Delta_\sp -\nu_\sm\Delta_\sm = -\lambda \Delta_\sp +O(\lambda^{1/2})
\quad 
\text{as}
\quad 
|\lambda|\to\infty\,.
\]
\end{lemma}

\begin{proof}
For $T$ as in equation \eqref{depcon}, we have evidently
\begin{align*}
D_\mathrm{E}(\lambda)	&=\det(W_1,W_2,W_3)|_{x=0}  \\
&=\det(T\hat Z_1,T\hat Z_2,T\hat Z_3)|_{y=0} \\
&=\det T(0)\det(\hat Z_1,\hat Z_2,\hat Z_3)|_{y=0}\\
& =C(\lambda)\det T(0)D_\mathrm{L}(\lambda)\,,
\end{align*}
where $C(\lambda)$ is the product of the ratios between bases $\hat Z_j=T^{-1}W_j$ of stable and unstable manifolds
and the basis elements $Z_j\sim \me^{\nu_j y} U_j$ prescribed in the definition of the Lagrangian Evans function,
or, equivalently, of basis elements $W_j\sim \me^{\mu_j x}V_j$ and $TZ_j$.

By Lemma \ref{katolem}, $ T Z_j \sim \me^{\nu_j y(x)} V_j, $
whereas, by equation \eqref{AB}, $\mu_j = (\dif y/\dif x)|_{\pm \infty}\nu_j$.
Thus, the ratio $ |TZ_j|/|W_j|$ is given by 
$$
\exp\big(\nu_j \lim_{x\to \pm \infty}(y(x)- x (\dif y/\dif x))\big)\,.
$$
Using $y(0)=0$, we obtain $y(x)=\int_0^x (\dif y/\dif x)(z)\,\dif z$, hence 
\[
y(x)-x (\dif y/\dif x)= \int_0^x \big((\dif y/\dif x)(z)-(\dif y/\dif x)(x)\big)dz.
\]
Substituting $\dif y/\dif x=\bar \rho(x)$, and taking the limit as $x\to \pm \infty$, we obtain the result.
The asymptotics for $\nu_\spm$ are readily obtained by spectral perturbation analysis, or by asymptotic analysis
of the characteristic polynomials of $\mat{B}_\spm $, in the limit as $|\lambda|\to \infty$.
\end{proof}
%

For the chosen pressure law and parameters, $\bar\rho$ is increasing, hence $\Delta_\spm <0$.
Moreover, $\Delta_\sp<0$ and $\Delta_\sm>0$, hence 
$D_\mathrm{E}(\lambda)/ D_\mathrm{L}(\lambda) \sim \me^{ \lambda |\Delta_+|}$, explaining the large difference
in winding between images of semicircular contours of large radius under $D_{\mathrm{E}}$ vs. $D_{\mathrm{L}}$.

\subsection{High-frequency asymptotics}
Lemma \ref{evrel} and the conclusion above explain the large difference between Eulerian and Lagrangian Evans functions, 
by a factor of order $\me^{C\lambda}$ as $|\lambda|\to\infty$.
However, they do not explain the ``goodness'' of the Lagrangian version. 
For this, we appeal to large-$\lambda$ asymptotics for the individual Evans function,
as carried out for the more difficult nonisentropic case in \cite{HLyZ}*{Prop. 4.2}, which shows that
\be\label{good}
D_\mathrm{L}(\lambda)\sim \me^{C\sqrt{\lambda}}
\quad
\text{as}
\quad
|\lambda|\to \infty\,.
\ee
A similar analysis carried out for the Eulerian Evans function gives 
\be\label{bad}
D_\mathrm{E}(\lambda)\sim \me^{C_2\lambda }
\quad
\text{as}
\quad
|\lambda|\to \infty\,,
\ee
in agreement with Lemma \ref{evrel}.
This verifies rigorously the observed
phenomenon that the Lagrangian Evans function indeed has much better behavior than the Eulerian version.

More important for our purposes is the asymptotic argument behind the result, which shows that, to leading order
as $|\lambda|\to \infty$, the basis elements $Z_j$ ``track'' the eigendirections of the frozen-coefficient matrix
$\mat{B}(y,\lambda)$ as $y$ is varied.  Thus, their magnitudes $r_j$ obey the simple scalar equations
$\dif r_j/\dif y= \nu_j(y) r_j$, where $\nu_j(y)$ are the eigenvalues of the frozen-coefficient matrix $\mat{A}(y,\lambda)$,
which, taking into account the prescribed asymptotics $r_j\sim \me^{\nu_j y}$ as $y\to \pm \infty$
results in a magnitude at $y=0$ of order $\me^{\int_{\pm \infty}^0 (\nu_j(y)- \nu_j(\pm\infty) dy}$ for each mode.

Among the $\nu_j$, there are two harmless ``parabolic'' modes $\mu_j \sim \sqrt{\lambda/\bar \tau}$,
giving combined contribution $\sim e^{C\sqrt{\lambda}}$.
The third, potentially harmful, mode is the ``hyperbolic'' mode associated with the density equation
$\lambda \tau + \tau'= u'$, whose principal part $\lambda \tau= -\tau'$, leads to the eigenvalue
$$
\nu_*(y) = -\lambda + O(\lambda^{1/2}).
$$
The crucial feature of this eigenvalue is that it is {\it to leading order constant in $y$.}
Thus, the associated mode $Z_*$
contributes to the Evans function magnitude 
$\me^{\int_{\pm \infty}^0 (\nu_*(y)- \nu_*(\pm\infty) dy}\sim \me^{C\sqrt{\lambda}}$ as $|\lambda|\to \infty$
of the same asymptotic order as the parabolic modes.

For the Eulerian Evans function, on the other hand, the corresponding hyperbolic mode $W_*$ satisfies to leading order
the scalar ODE
$\lambda \rho + \bar u \rho=0$, with an associated eigenvalue 
$$
\mu_*(x)= -\lambda/\bar u(x) + O(\sqrt{\lambda})= -(\bar \rho(x)/m)\lambda + O(\sqrt{\lambda})
$$
that is {\it variable coefficient to leading order} in $x$.
This leads to a factor $\sim \me^{C_1\lambda}$ in the Eulerian Evans function, and the resulting $\me^{C_1\lambda}$
asymptotics cited above.
	

\section{Pseudo-Lagrangian coordinates: multiple space dimensions}\label{sec:pseudo}
We turn now to the multidimensional case. We consider the 
isentropic Navier--Stokes equations in space dimension $d=2$. In Eulerian coordinates, the system takes the form,
in Eulerian coordinates:
\begin{subequations}\label{eq:dns}
\beq
\partial_t\rho+\dv(\rho\vec{v})=0\,,
\eeq
\beq
\partial_t(\rho\vec{v})+\dv(\rho \vec{v}\otimes\vec{v})+\grad p=\mu\Delta\vec{v}+(\mu+\eta)\grad\dv\vec{v}\,,
\eeq
\end{subequations}
where $\rho$ is density, $\vec{v}=(v_1,v_2)$ velocity, $p$ pressure, related to density by equation \eqref{eq:pressure},
and constants $\mu$ and $\eta$ are coefficients of first and second viscosity \cites{Ba,HLyZ3}.
Linearizing about a steady planar profile $(\rho, \vec{v})=(\bar \rho, \overline{\vec{v}})(x_1)$ varying 
in the $x_1$ direction only,
without loss of generality $\bar v_2\equiv 0$, we obtain the eigenvalue equations
\begin{subequations}\label{eq:dnseval}
\beq
	\lambda \rho+ \dv(\bar \rho\vec{v}+ \rho \overline{\vec{v}})=0\,,
\eeq
\begin{multline}
	\lambda (\bar \rho\vec{v}+ \rho \overline{\vec{v}}) 
	\dv(\rho\bar{\vec{v}}\otimes\bar{\vec{v}}+\bar\rho\vec{v}\otimes \bar{\vec{v}}+\bar{\rho}\bar{\vec{v}}\otimes\vec{v})
	+ \grad p(\bar\rho)
	\\=\mu\Delta\vec{v}+(\mu+\eta)\grad\dv\vec{v}\,.
\end{multline}
\end{subequations}
Taking the Fourier transform in $x_2$, we obtain a family of ordinary differential equations in $x_1$ parametrized by the Fourier frequency $\xi$.
Expressing this as a first-order system, we may define an Evans function 
\be\label{multievans}
D_\mathrm{E}(\lam,\xi)
\ee
similarly as in the one-dimensional case, with zeros corresponding to generalized eigenmodes $\me^{\mi\xi x_2}w(x_1)$,
$w$ decaying at infinity, associated with eigenvalue $\lambda$. 
See \cite{HLyZ3,BHLZ2} for further details.

This Evans function has equally poor behavior as the one-dimensional version; indeed, for $\xi=0$, the multidimensional Eulerian 
Evans function reduces to (a nonvanishing multiple of) the one-dimensional one.
However, in contrast to the one-dimensional case, a useful Lagrangian version of the Evans function does not
seem to be available; Pogan, Yao, \& Zumbrun \cite{PYZ} discuss this issue in some depth.

\subsection{Pseudo-Lagrangian Coordinates}
To resolve this problem, making possible practical multidimensional Evans function computations, we introduce instead
a new {\it pseudo-Lagrangian} formulation of the Evans function, based on the Eulerian version, but sharing the
good properties of the one-dimensional Lagrangian Evans function.
Namely, dropping the subscript on $x_1$, and writing the first-order Evans system as
$$
\dif W/\dif x=\mat{A}(x;\lambda, \xi) W,\,
$$
we introduce $\dif Y/\dif y= \mat{B}(y; \lambda, \xi)Y,$ where $\mat{B}$ is defined by $\mat{B}(y(x);\lambda,\xi)=(\dif x/\dif y)\mat{A}(x;\lambda,\xi)$,
and denote the resulting Evans function by $D_\mathrm{pL}(\lam,\xi)$.

Partial justification for this choice is given by the following straightforward result.
Abusing notation somewhat, let $D_\mathrm{pL}(\lambda)$ denote the one-dimensional version of the pseudo-Lagrangian Evans function,
obtained from the Eulerian Evans system by the change of dependent variable $\dif y/\dif x=\bar\rho(x)$ as was done in the multidimensional case.

\begin{proposition}
The one-dimensional pseudo-Lagrangian Evans function $D_\mathrm{pL}(\lambda)$
agrees with (i.e., is a constant multiple) of
the one-dimensional Lagrangian Evans function $D_\mathrm{L}(\lambda)$.
\end{proposition}
\begin{proof}
	This follows by the argument in the proof of Lemma \ref{evrel}, but now observing that the Lagrangian
	and pseudo-Lagrangian flows are conjugate by a change of dependent variables alone, with no change of independent
	variable.
\end{proof}

Further motivation is given by the hyperbolic $\rho$ equation of the Fourier transformed eigenvalue equation,
$$
\lambda \rho+ (\dif /\dif x)(\rho \bar v_1+ \bar \rho v_1) + \mi\xi \bar \rho v_2=0\,,
$$
which has principal part $ \lambda \rho+ \bar v_1 (\dif/\dif x)(\rho)  =0$,
or $\dif \rho/\dif x= -\lambda/\bar v_1$ as in the one-dimensional case.
Thus, $\dif \rho/\dif y= (\dif \rho/\dif x)(\dif x/\dif y)= -(\lambda/m)\rho$, with $m\equiv \bar \rho \bar v_1$ constant,
similarly as in the one-dimensional case.
Thus, the corresponding asymptotic eigenvalue $\nu_*(\lambda, \xi)$ of the frozen-coefficient matrix
$\mat{B}(y;\lam,\xi)$ is, to leading order, independent of $y$, and we obtain favorable large-$|\lambda|$ asymptotics
also for the multidimensional version of the pseudo-Lagrangian Evans function.

%
%


\subsection{Numerical performance}\label{ssec:num2}
As in one dimension, we find that the image of a contour under evaluation of the multidimensional Evans function in Eulerian coordinates raps excessively around the origin before unwinding again and varies in modulus significantly more than when using pseudo-Lagrangian coordinates. For example, when $\gamma = 5/3$, $u_\sp= 0.06$, $\xi = 1$,
and we compute the Evans function on a contour like that shown in \ref{fig:compare}(b), but with inner radius set to $r=1e-3$ and outer radius to $R =30$, we find that in Eulerian coordinates it takes 1344 points on the pre-image contour in order for the image contour to vary in relative distance no more than 0.2, whereas for pseudo-Lagrangian coordinates, 212 pre-image points suffice. 
As seen in Figure \ref{fig:eulerian_evans}, the Evans function computed in Eulerian coordinates varies in modulus over three times more orders of magnitude then in pseudo-Lagrangian coordinates. 
\begin{figure}[ht!]
 \begin{center}
$
\begin{array}{lcr}
\text{(a)} \includegraphics[scale=0.23]{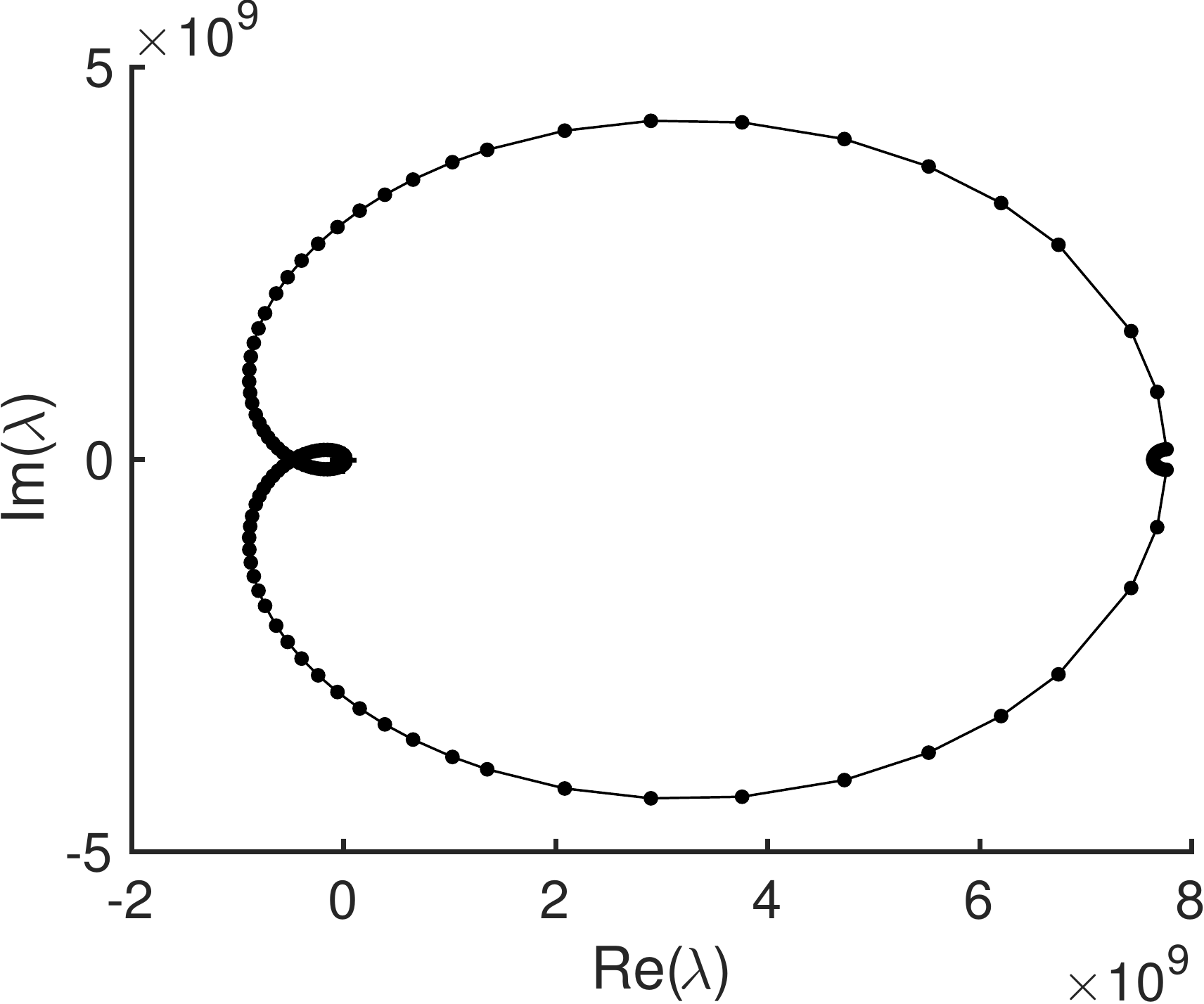} & \text{(b)} \includegraphics[scale=0.23]{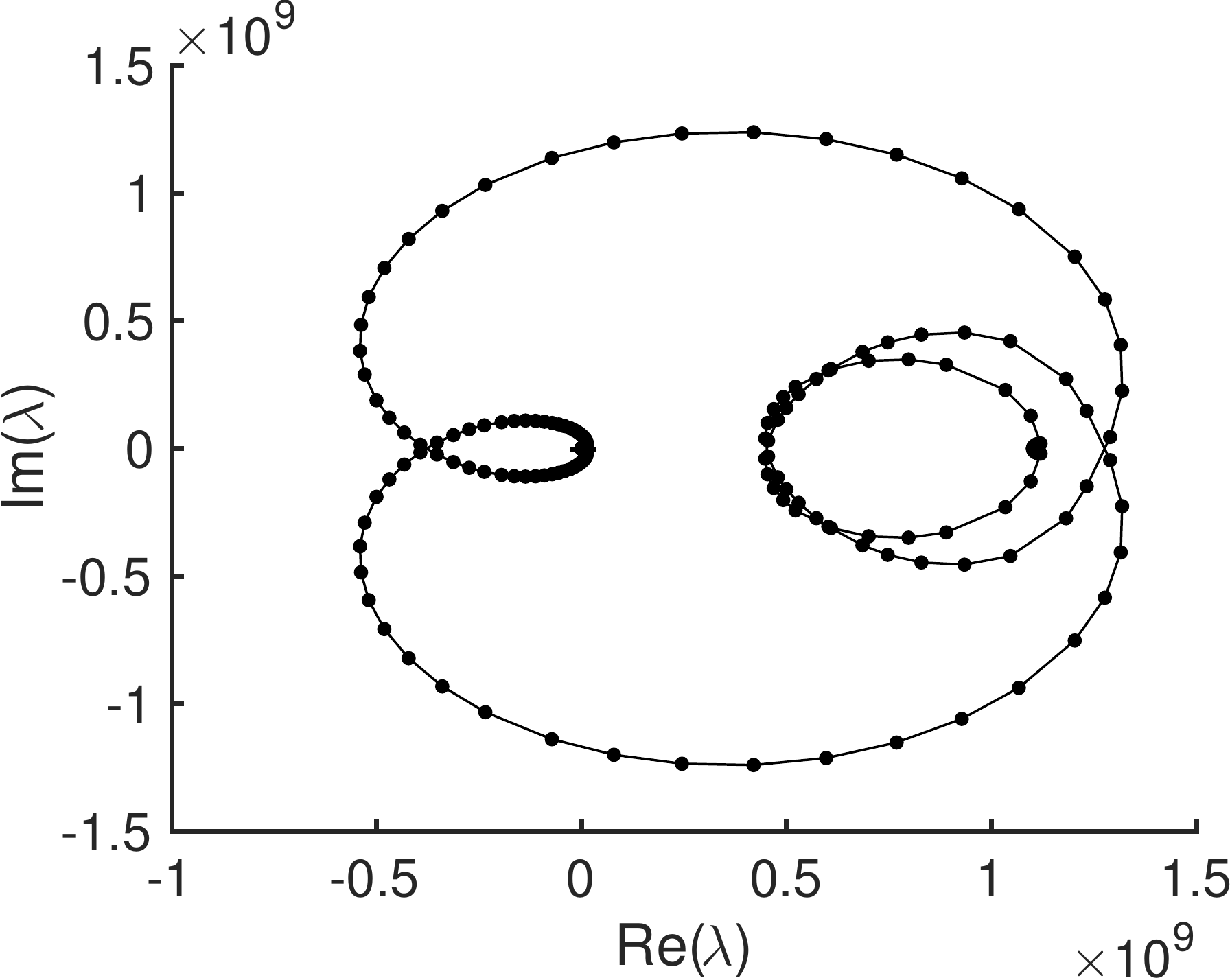}   & \text{(c)} \includegraphics[scale=0.23]{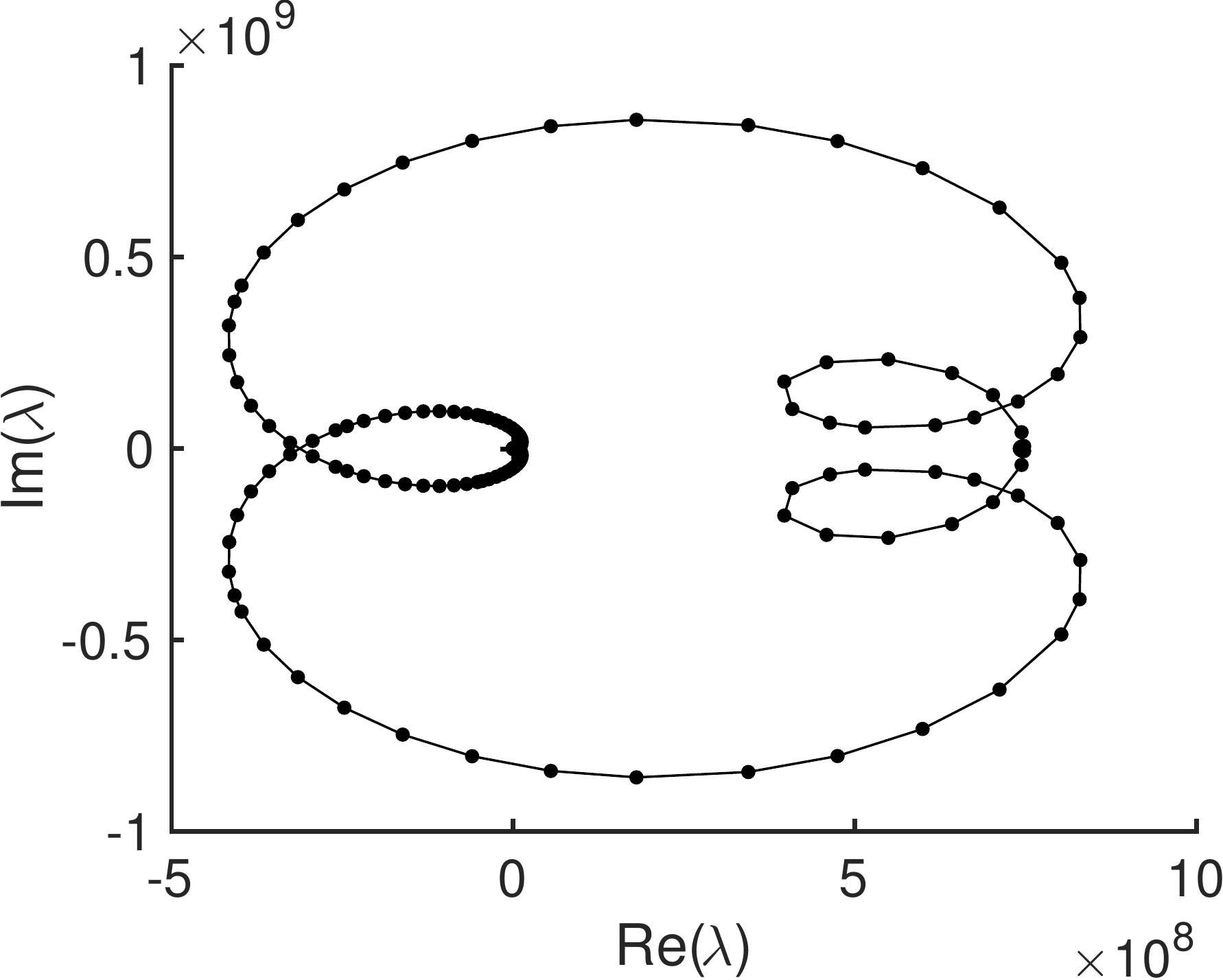} \\
 \text{(d)} \includegraphics[scale=0.23]{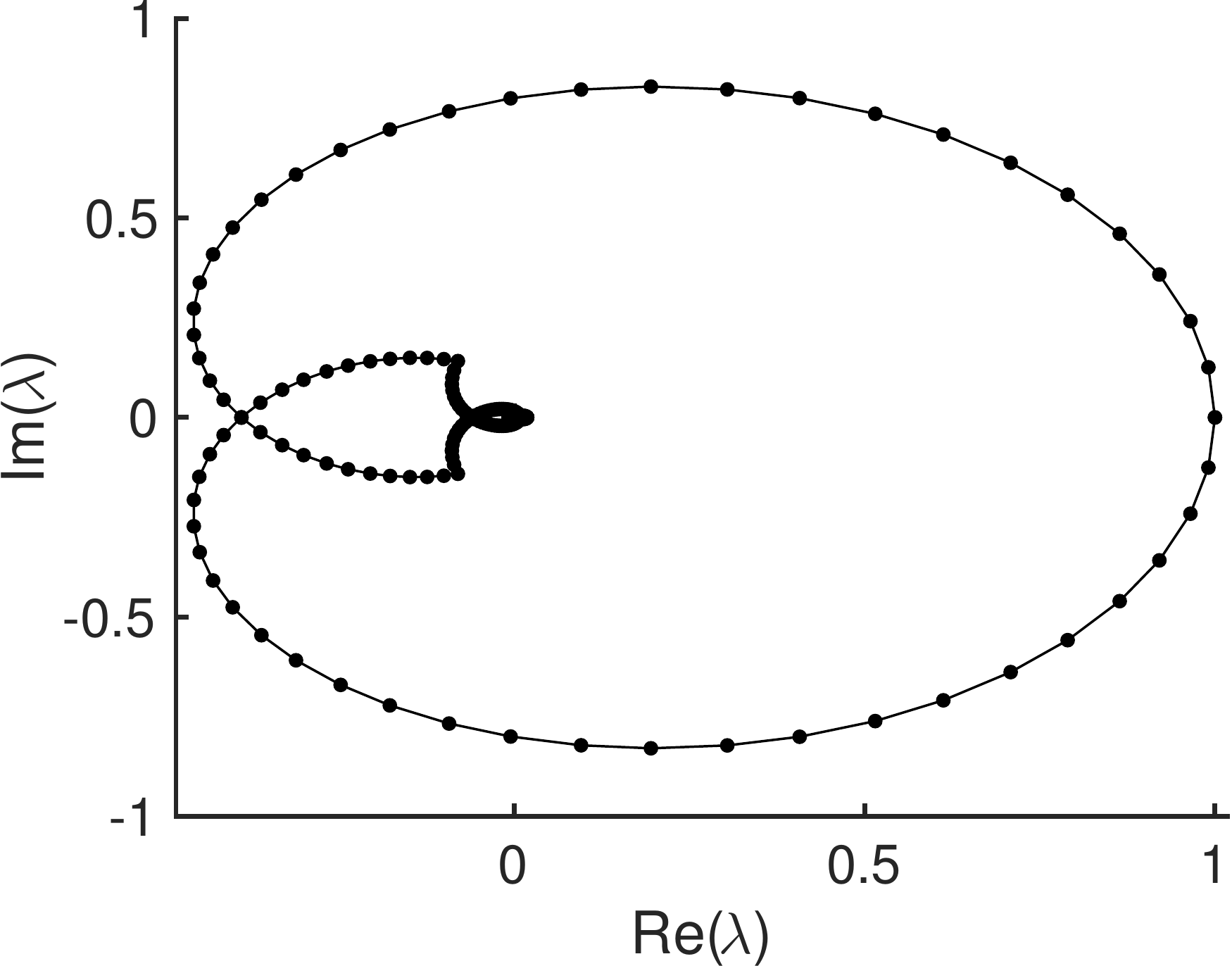} & \text{(e)} \includegraphics[scale=0.23]{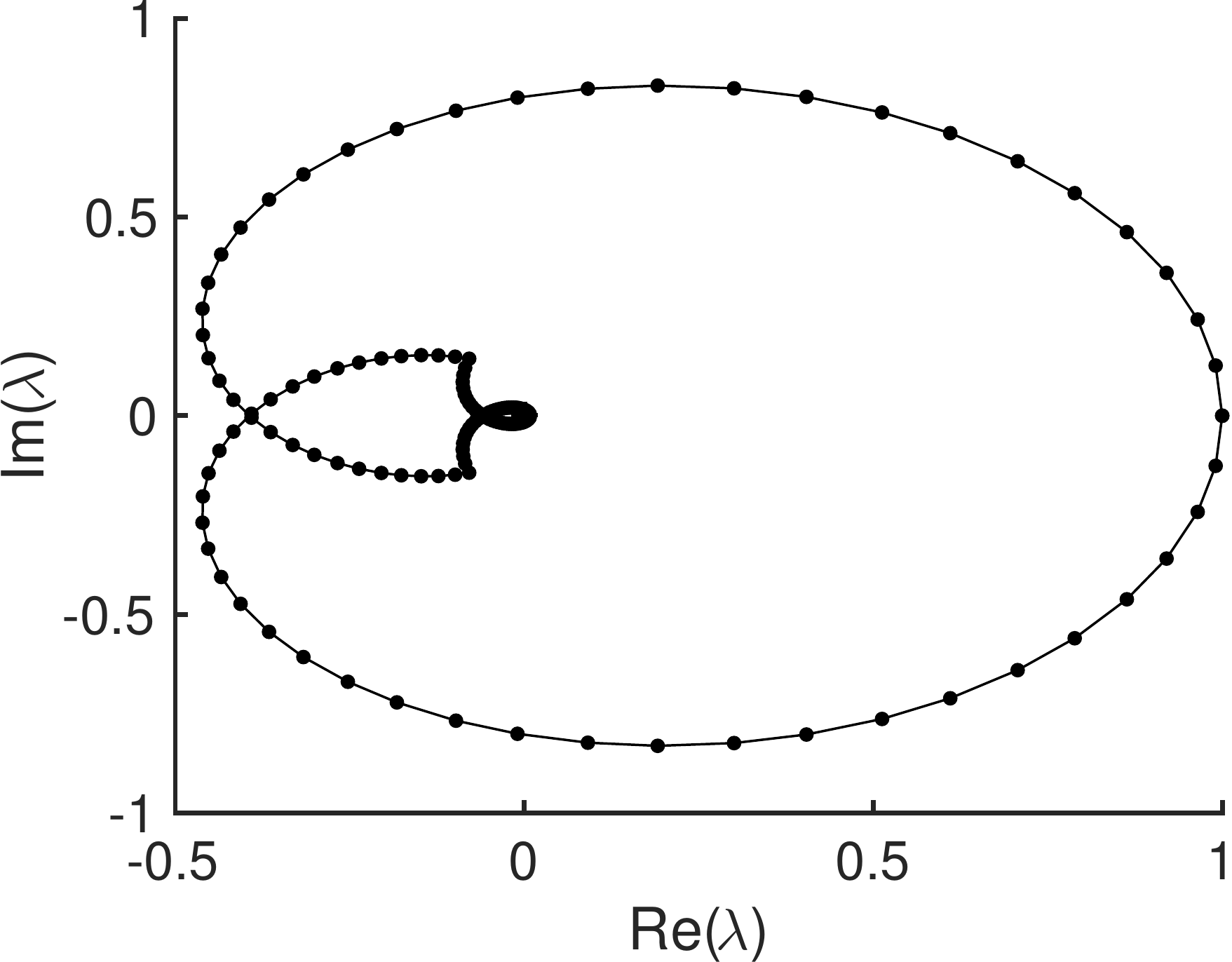} & \text{(f)}  \includegraphics[scale=0.23]{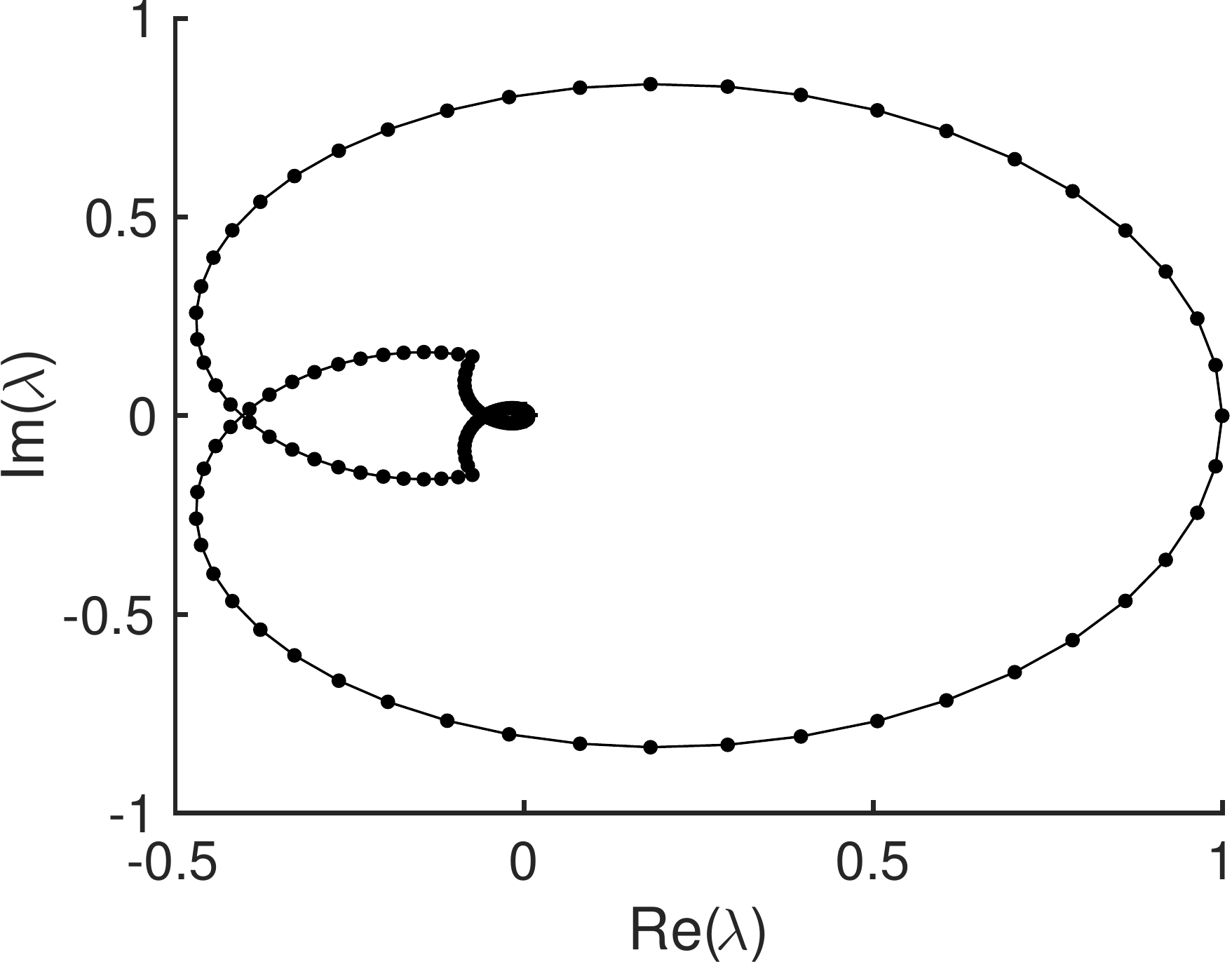} 
\end{array}
$
\end{center}
\caption{Plot of the Evans function for two-dimensional isentropic gas in Eulerian coordinates (top row) and pseudo-Lagrangian coordinates (bottom row) when $\gamma = 5/3$ and $u_\sp = 0.06$. In all figures, the Evans function is evaluated on a semi-annulus contour with inner radius $r=10^{-3}$ and outer radius $R = 30$. The Fourier parameter is $\xi=0$ in (a) and (c), $\xi = 0.5$ in (b) and (e), and $\xi = 1$ in (c) and (f).   }
\label{fig:eulerian_evans}
\end{figure}
An even starker contrast occurs when the Evans function is computed on a contour with outer radius $R = 90$ and inner radius $r = 1e-3$, but with $\gamma = 5/3$, $\xi = 0$, and $u_\sp = 0.001$. The Evans function in Eulerian coordinates takes 4.06 days to compute, varies over 225 orders of magnitude, and requires 12,708 points in order for the image contour to vary in relative distance no more than 0.2, whereas the Evans function in pseudo-Lagrangian coordinates takes 20.4 minutes to compute, varies over 12 orders of magnitude, and requires 740 points.
Furthermore, in pseudo-Lagrangian coordinates, the Evans function for multidimensional isentropic gas has small variation as $\xi$ varies. This is  shown in Figure \ref{fig94}.
\begin{figure}[ht!]
 \begin{center}
$
\begin{array}{lcr}
\includegraphics[scale=0.5]{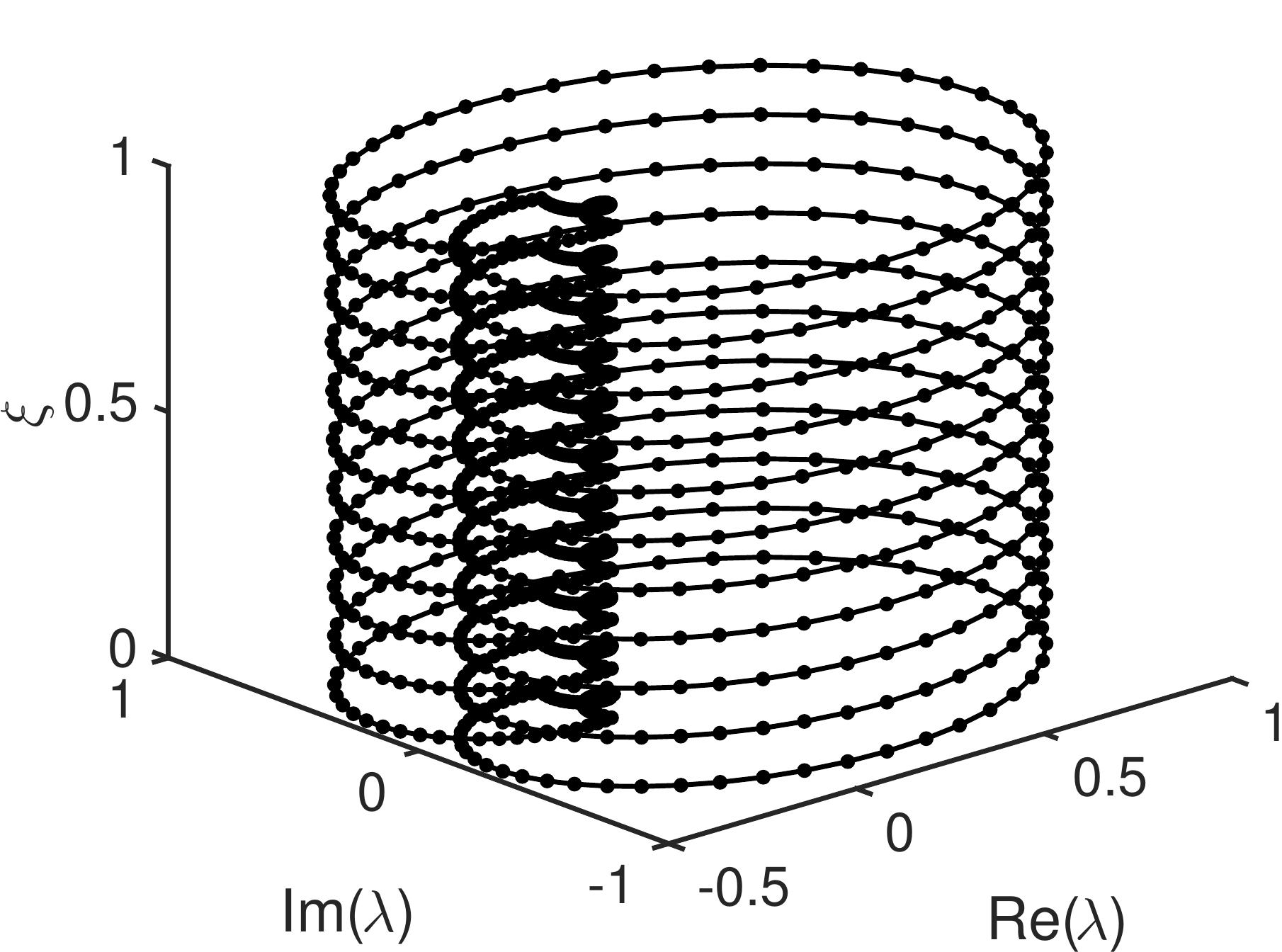}
\end{array}
$
\end{center}
\caption{Plot of the Evans function for two-dimensional isentropic gas in pseudo-Lagrangian coordinates when $\gamma = 5/3$, $u_\sp = 0.06$, and $\xi$ varying between $0$ and $1$ by $0.1$.  }
\label{fig94}
\end{figure}
Finally, differences in cost (in terms of the number of evaluations required and the computing time) are compiled in Table \ref{tab:cost}.
The same significant improved performance of pseudo-Lagrangian coordinates manifests itself in the full gas system as well \cite{HLyZ3}. One other advantage of pseudo-Lagrangian coordinates for isentropic gas is that it requires the traveling wave profile in Lagrangian coordinates and not Eulerian coordinates. For the examples featured in Figure \ref{fig:eulerian_evans}, we had to use continuation to solve the profile in Eulerian coordinates as $u_\sp$ decreased, and eventually solve it as a scalar system using a stiff ODE solver. On the other hand, in Lagrangian coordinates continuation was not needed to solve the profile.

%
%
%

\begin{table}[ht!]
\caption{The computational cost of the Eulerian method and pseudo-Lagrangian method.
The first two columns indicate the parameter $\tau_\sp$ and the Fourier variable $\xi$. The last four columns indicate the number of points and computation time it took to compute the Evans function on a contour of radius $R = 90$ with an adaptive Evans-function evaluator which requires that the relative error between points in the image contour be no greater than 0.2. In the last four columns a $p$ represents the number of points on which the contour is computed and a $t$ represent the computation time. The subscripts $\mathrm{E}$ and $\mathrm{pL}$ stand respectively for the Euler method and the pseudo-Lagrange method.}
\label{tab:cost}
\begin{tabular}{|c|c||c|c||c|c|}
\hline
$\tau_\sp$& $\xi$ & $p_\mathrm{E}$ & $t_\mathrm{E}$ & $p_\mathrm{pL}$ & $t_\mathrm{pL}$ \\
\hline
\hline
0.2733&    0&  238& 579.9&  112&  396\\ 
\hline
0.2733&  0.3&  256& 466.9&  138& 368.2\\ 
\hline
0.2733&  0.6&  240& 458.1&  124& 357.7\\ 
\hline
0.22&    0&  348& 983.7&  120& 386.8\\ 
\hline
0.22&  0.3&  376& 785.7&  150& 353.1\\ 
\hline
0.22&  0.6&  356& 775.1&  136&  358\\ 
\hline
0.1667&    0&  396& 1163&  184& 530.1\\ 
\hline
0.1667&  0.3&  440& 837.2&  218& 504.1\\ 
\hline
0.1667&  0.6&  422& 862.8&  200& 498.4\\ 
\hline
0.1133&    0&  676& 1829&  206& 555.3\\ 
\hline
0.1133&  0.3&  702& 1612&  248& 479.1\\ 
\hline
0.1133&  0.6&  676& 1641&  226& 486.1\\ 
\hline
0.06&    0&  948& 2916&  340& 814.6\\ 
\hline
0.06&  0.3& 1012& 2566&  392& 758.7\\ 
\hline
0.06&  0.6&  984& 2577&  378& 782.1\\
\hline
0.001& 0 & 12708 & 3.51e5 & 740 & 1224 \\
\hline
\end{tabular}
 \end{table}

\section{Conclusions}\label{sec:conclusions}
Our results illustrate that coordinate choices, at the level of physical models, can have substantial impacts on the viability of a given Evans-function computation. Coupling this with our companion results \cite{BHLZ2} on the practical role of coordinate choices in the construction of the first-order eigenvalue equation, we see a simple takeaway message: coordinate choices matter. While viscous shock profiles in one space dimension can equally well be described in Eulerian and Lagrangian coordinates, the two Evans functions arising from the two models behave substantially differently, and these differences affect the viability of computations for spectral stability.

The import of coordinate choices goes far beyond minimizing winding for attractive pictures of Evans-function output. For physical systems with many parameters and/or for multidimensional problems, it is essential to minimize the number of function evaluations to have a chance to properly explore parameter and frequency space. Indeed, as noted above, we expect pseudo-Lagrangian coordinates to be necessary for any kind of computational Evans-function analysis of multidimensional problems in magnetohydrodynamics and in detonation theory. More generally, this phenomenon will be present in general composite type hyperbolic--parabolic systems and perhaps in other settings as well. While the phenomenon is not present in the 2nd-order strictly parabolic case\footnote{This explains why issues such as the above have not appeared in the extensive Evans-function literature associated with traveling-wave solutions of reaction-diffusion equations.}, it does suggest an interesting and important open problem. That is, given a physical system, which representative of the Evans function is the ``best'' for computational purposes? Since the stability calculations generally involve winding numbers, one measure of ``best'' might be in terms of minimizing winding. 
Certainly the example of gas dynamics presented here suggests that some kind of answer to the above question is required if numerical Evans-function calculations are going to be part of a general purpose, push-button stability calculator. Thus, in addition to recent developments of Evans-function approximations in numerical proofs of stability \cites{B,BZ}, we see optimizing the computed Evans functions as a central issue in the future development of computational Evans-function techniques.

%
%
%
%
%
%
%
%

%


\end{document}